\documentclass{amsart}
\usepackage{amsmath}
\usepackage{mathrsfs}
\usepackage{amsfonts}
\usepackage{amsthm}
\usepackage{amssymb}
\usepackage{amscd}
\usepackage{graphicx}
\input xy
\xyoption{all} \theoremstyle{plain}
\newtheorem*{theo}{Theorem}
\newtheorem{theorem}{Theorem}[section]
\newtheorem{lemma}[theorem]{Lemma}
\newtheorem{definition}[theorem]{\bf Definition}
\newtheorem{proposition}[theorem]{\bf Proposition}

\newtheorem{remark}[theorem]{Remark}
\newtheorem{parrafo}[theorem]{}
\newcommand{\grad}{\textnormal{grad}}
\newcommand{\Gr}{\operatorname{Gr}}
\newcommand{\rank}{\textnormal{rank}}
\newcommand{\codim}{\textnormal{codim}}
\newcommand{\Jac}{\textnormal{Jac}}
\newcommand{\End}{\textnormal{End}}
\newcommand{\Lie}{\textnormal{Lie}}
\newcommand{\Hom}{\textnormal{Hom}}
\newcommand{\Aut}{\textnormal{Aut}}
\newcommand{\Conv}{\textnormal{Conv}}

\newcommand{\Stab}{\textnormal{Stab}}
\newcommand{\diag}{\textnormal{diag}}
\newcommand{\Sym}{\textnormal{Sym}}
\newcommand{\Ant}{\textnormal{Ant}}

\begin{document}
\title[The Hodge--Poincar\'e polynomial of the moduli spaces of
stable vector bundles] {The Hodge--Poincar\'e polynomial of the
moduli spaces of stable vector bundles over an algebraic curve}
\thanks{This work has been partially supported by a EC Training Fellowship,
within the Research Training Network LIEGRITS: Flags, Quivers and
Invariant Theory in Lie Representation Theory, which is a Marie
Curie Research Training Network funded by the European community
as project MRTN-CT 2003-505078} \subjclass{14H60, 14D20, 14F45}
\keywords{Moduli spaces, vector bundles, Hodge--Poincar\'e
polynomials}

\author{Cristian Gonz\'alez-Mart\'inez}
\address{{\it Formerly at:} Department of Mathematics\\Tufts University\\Bromfield-Pearson Building\\503 Boston
Avenue\\
Medford, MA 02155\\USA.}
\address{{\it Currently at:} DG-Payments and Market Infrastructure\\European Central Bank\\Neue Mainzer Stra$\beta$e\\ 60311 Frankfurt am Main\\Germany.}
\email{c.gonzalez-martinez@hotmail.com}

\maketitle

\begin{abstract}Let X be a nonsingular complex projective variety that
is acted on by a reductive group $G$ and such that $X^{ss} \neq
X_{(0)}^{s}\neq \emptyset$. We give formulae for the
Hodge--Poincar\'e series of the quotient $X_{(0)}^s/G$. We use
these computations to obtain the corresponding formulae for the
Hodge--Poincar\'e polynomial of the moduli space of properly stable vector
bundles when the rank and the degree are not coprime. We compute
explicitly the case in which the rank equals 2 and the degree is
even.\end{abstract}



\section{Introduction and statement of results}

Let $\mathcal{M}(n,d)$ be the moduli space of semistable vector bundles of rank $n$ and degree $d$. The cohomology of  $\mathcal{M}(n,d)$ has been of
great interest to a large number of mathematicians for the last
40 years. If we denote by $\mathcal{M}_{(0)}^s(n,d)$ the moduli
space of (properly) stable vector bundles, it is not difficult to
see that when $(n,d)=1$ then $\mathcal{M}(n,d)=
\mathcal{M}_{(0)}^s(n,d)$.

The first results on the cohomology of $\mathcal{M}(n,d)$ are due
to P. E. Newstead who computed the Betti numbers of
$\mathcal{M}(2,1)= \mathcal{M}_{(0)}^s(2,1)$ from the results
obtained in his paper \cite{N1}. From these results Harder
observed that the Betti numbers of $\mathcal{M}(2,1)$ can also be
computed by arithmetic methods and the Weil conjectures \cite{H}.
The latter method was generalized by Harder and Narasimhan to
obtain the Betti numbers of $\mathcal{M}(n,d)$ when $(n,d)=1$ (see
\cite{HN}).

Another way to carry out these computations was introduced by
Atiyah and Bott in their seminal paper \cite{AB}. This is based on
the definition of a stratification in the infinite dimensional
space of all possible holomorphic structures on a fixed
$\mathcal{C}^{\infty}$ bundle of rank $n$ and degree $d$. The
stratification turns out to be equivariantly perfect with respect
to the action of the gauge group that is acting on the space of
all possible holomorphic structures. Then, the equivariant Morse
inequalities deduced from that allow us to obtain an inductive
formula for the equivariant Betti numbers of the stratum
classifying the semistable points. From this, when $(n,d)=1$, they
obtain a formula for the Betti numbers of $\mathcal{M}(n,d)$ by
representing this moduli space as a geometric invariant theory (GIT)
quotient of the space of all possible holomorphic structures by
the action of the group of complex automorphisms or complexified
gauge group.

There is still another way of doing this that is described in \cite{K1}. Except for the fact that Atiyah and Bott's method
for computing the Betti numbers works with infinite-dimensional
spaces and groups, the method of \cite{K1} can be regarded as a
generalization of that of \cite{AB}.

Note that there are several ways of representing
$\mathcal{M}(n,d)$ as a geometric invariant theory quotient, in
this case, as the quotient of a nonsingular projective algebraic
variety by the action of an algebraic reductive group. One can
look at the cohomology of the GIT quotient $X/\!/G$ of a nonsingular
projective algebraic variety $X$ acted on by an algebraic
reductive group $G$. Let $X^{ss}$ and $X_{(0)}^{s}$ denote the set
of semistable and properly stable points for the action of $G$.
When $X_{(0)}^s=X^{ss}$ (note that for $\mathcal{M}(n,d)$ this
corresponds to $\mathcal{M}(n,d) = \mathcal{M}_{(0)}^s(n,d)$ or
what is the same, when $(n,d)=1$), in \cite{K1} the Poincar\'e
polynomial of $X/\!/G$ is computed bearing in mind that there is a
natural identification between the cohomology of $X/\!/G$ with
rational coefficients and the equivariant cohomology of $X^{ss}$
with rational coefficients.
\begin{parrafo}\label{Introduccion stratification}\textnormal{Actually, it is proved that
there is a stratification $\{ \mathcal{S}_{\beta}:\beta \in
\mathcal{B}\}$ of $X$ by nonsingular $G$-invariant locally closed
subvarieties $\mathcal{S}_{\beta}$ satisfying the following
properties (see \cite{K1}):
\begin{itemize}\item[(i)]$X^{ss}$ coincides with the unique open
stratum $\mathcal{S}_0$.\item[(ii)]The equivariant Morse
inequalities are equalities, that is
$$P^G_t(X)=P^G_t(X^{ss})+\sum_{\beta \neq 0}t^{2\codim
\mathcal{S}_{\beta}}P^{G}_t(S_{\beta}).$$If $H_G^{\ast}(Y)$
denotes the $G$-equivariant cohomology ring of $Y$, then
$P_t^G(Y)=\sum_{i}t^i \dim H_G^i(Y)$ is the equivariant Poincar\'e
series of $Y$. Here $\codim \mathcal{S}_{\beta}$ denotes the complex
codimension of $\mathcal{S}_{\beta}$ in $X$. In general one has to
take care of the possibility of the strata being not connected but
for our purposes we don't need to consider this. \item[(iii)]If
$\beta \neq 0$ there is a proper nonsingular subvariety
$Z_{\beta}$ of $X$ invariant under the action of a reductive
subgroup Stab$\beta$ of $G$ such that
$$H^{\ast}_G (\mathcal{S}_{\beta})\cong H^{\ast}_{\Stab \beta}
(Z_{\beta}^{ss}),$$where $Z_{\beta}^{ss}$ is the semistable
stratum of $Z_{\beta}$ under the action of $\Stab \beta$ appropriately
linearized.
\end{itemize}}\end{parrafo}

Under the hypothesis of $X_{(0)}^s=X^{ss}$ one has that
$P^G_t(X^{ss})=P_t(X/\!/G)$ where $P_t$ is the usual Poincar\'e
polynomial. Then from the identity in (ii) one may obtain a formula that computes the Betti numbers of $X/\!/G$ when
$X_{(0)}^s=X^{ss}$ (see \cite[Theorem 8.12]{K1}).

The stratification $\{ \mathcal{S}_{\beta} :\beta \in \mathcal{B}
\}$ can be defined either using the moment map and symplectic
geometry or algebraically. The indexing set $\mathcal{B}$ is going
to be given by a finite set of points in a positive Weyl chamber
of the Lie algebra of a maximal compact torus $T$ of $G$.

Our research is focussed on the cohomology of the moduli space of
properly stable vector bundles $\mathcal{M}_{(0)}^s(n,d)$ when
$(n,d)\neq 1$. If we represent $\mathcal{M}_{(0)}^s(n,d)$ as a GIT
quotient $X/\!/G$, the condition $(n,d)\neq 1$ implies that
$X^{ss}\neq X_{(0)}^{s}\neq \emptyset$ or what is the same, there
are semistable points in $X$ that are not properly stable.

When $X^{ss}\neq X_{(0)}^{s}\neq \emptyset$ it could happen that
$X/\!/G$ may have serious singularities. In \cite{K2} F. Kirwan
shows a way of blowing up the variety $X$ along a sequence of
nonsingular subvarieties to obtain a variety $\widetilde{X}$ with
a linear action of $G$ such that
$\widetilde{X}^{ss}=\widetilde{X}_{(0)}^s$. Then,
$\widetilde{X}/\!/G$ can be regarded as a ``partial'' resolution of
singularities of ${X}/\!/G$ in the sense that the most serious
singularities of ${X}/\!/G$ have been resolved. This may be used to
compute the Betti numbers of $\widetilde{X}/\!/G$ in terms of those
of ${X}/\!/G$ and the dimensions of the rational intersection
homology groups of ${X}/\!/G$ in terms of Betti numbers of the
partial desingularisations $\widetilde{X}/\!/G$. In \cite{K4} these
techniques are applied to the case of the moduli space of
semistable vector bundles of rank $n$ and degree $d$ when
$(n,d)\neq 1$.

From this resolution of singularities, in \cite{K6} the
stratification $\{ \mathcal{S}_{\beta}:\beta \in \mathcal{B}\}$ is
refined to obtain a stratification of the set $X^{ss}$ of
semistable points, so that the set $X_{(0)}^s$ of properly stable
points is an open stratum. This new stratification is not
equivariantly perfect hence one may not expect to obtain nice
formulae for the Betti numbers of the geometric quotient
$X_{(0)}^s/G$.

In this paper we use Deligne's extension of Hodge theory together
with the previous refined stratification in order to study the
Hodge--Poincar\'e series of a nonsingular projective variety that
is acted on by a reductive group and such that $X^{ss}\neq
X_{(0)}^{s}\neq \emptyset$. This can be carried out because of the
good properties of the Hodge--Deligne series and its relationship with
the Hodge--Poincar\'e series given by Poincar\'e duality. We obtain
formulae for these series. After that we use these
computations to obtain corresponding formulae for the
Hodge--Poincar\'e series of the moduli space of stable vector
bundles when the rank and the degree are not coprime. Finally, we
compute explicitly the Hodge--Poincar\'e polynomial of the moduli
space of (properly) stable vector bundles when the rank is 2 and
the degree is even. Using Poincar\'e duality one may obtain the Hodge--Deligne polynomial of $\mathcal{M}^s_{(0)}(2,d)$ when $d$ is even. The latter was first computed in \cite{MOV2}. The Hodge--Poincar\'e polynomial is given in the following
Theorem.
\begin{theo}The Hodge--Poincar\'e polynomial of $\mathcal{M}^s_{(0)}(2,d)$ for $d$ even is
given
by\begin{align*}HP(&\mathcal{M}^s_{(0)}(2,d))(u,v)=\frac{1}{2(1-uv)(1-u^2v^2)}\bigg{[}2(1+u)^g(1+v)^g
(1+u^2v)^g(1+uv^2)^g-\\&
-(uv)^{g-1}(1+u)^{2g}(1+v)^{2g}(2-(uv)^{g-1}+(uv)^{g+1})-(uv)^{2g-2}(1-u^2)^g(1-v^2)^g(1-uv)^2\bigg{]}.
\end{align*}\end{theo}

The layout of the paper is as follows. In Section \ref{Hodge} we give
an account of the results on Deligne's extension of Hodge theory, that we shall need
throughout this paper. In Section \ref{Stratifications} we recall the definitions and explicit descriptions of the Morse
stratification $\{ \mathcal{S}_{\beta} \}_{\beta \in \mathcal{B}}$
of $X$ and its refined stratification introduced by Kirwan in \cite{K1} and \cite{K6}. In Section \ref{Cohomological formulae}, from the
previous stratifications, we obtain formulae for the
Hodge--Poincar\'e series of the geometric quotient
$X_{(0)}^{s}/G$. These formulae are adapted in Section \ref{Cohomological formulae for the moduli space of stable vector bundles
when the rank and the degree are not coprime} to obtain
the corresponding ones for $\mathcal{M}_{(0)}^{s}(n,d)$ when
$(n,d)\neq 1$. Finally, in Section \ref{Cohomological formulae for the moduli space of stable vector bundles
when the rank is 2 and the degree is even}  we compute explicitly the
Hodge--Poincar\'e polynomial of $\mathcal{M}_{(0)}^{s}(2,d)$ when
$d$ is even.


\section{Hodge Theory}\label{Hodge}


We use Deligne's extension of Hodge theory which applies to
varieties which are not necessarily compact, projective or smooth
(see \cite{D2}, \cite{D3} and \cite{D4}). We start by giving a
review of the notions of pure Hodge structure, mixed Hodge
structure, Hodge--Deligne and Hodge--Poincar\'{e} polynomials.

\begin{definition}\textnormal{A \emph{pure Hodge structure of weight $m$} is given by
a finite dimensional $\mathbb{Q}$-vector space $H_{\mathbb{Q}}$
and a finite decreasing filtration $F^p$ of
$H=H_{\mathbb{Q}}\otimes \mathbb{C}$
$$H\supset \ldots \supset F^p \supset \ldots \supset ( 0 ) , $$
called \emph{the Hodge filtration}, such that $H=F^p\oplus
\overline{F^{m-p+1}}$ for all $p$. When $p+q=m$, if we set
$H^{p,q}=F^p \cap \overline{F^q}$, the condition $H=F^p\oplus
\overline{F^{m-p+1}}$ for all $p$ implies an equivalent definition
for a pure Hodge structure, that is, a decomposition
$H=\bigoplus_{p+q=m}H^{p,q}$
satisfying that $H^{p,q}=\overline{H^{q,p}}$, where
$\overline{H^{q,p}}$ is the complex conjugate of $H^{q,p}$. }
\end{definition}

\begin{definition}\label{MHS}\textnormal{A \emph{mixed Hodge structure} consists
of a finite dimensional $\mathbb{Q}$-vector space
$H_{\mathbb{Q}}$, an increasing filtration $W_l$ of
$H_{\mathbb{Q}}$, called \emph{the weight filtration}$$\ldots
\subset W_l \subset \ldots \subset H_{\mathbb{Q}},$$and the Hodge
filtration $F^p$ of $H=H_{\mathbb{Q}}\otimes \mathbb{C}$, where
the filtrations $F^p Gr_l^W$ induced by $F^p$ on $$Gr_l^W=(W_l
H_{\mathbb{Q}} / W_{l-1}H_{\mathbb{Q}})\otimes \mathbb{C} = W_l H
/ W_{l-1}H$$ give a pure Hodge structure of weight $l$. Here $F^p
Gr_l^W$ is given by $(W_lH\cap F^p +W_{l-1}H)/W_{l-1}H.$}
\end{definition}

Associated to the Hodge filtration and the weight filtration we
can consider the quotients $Gr_l^W=W_l/W_{l-1}$ of Definition
\ref{MHS}, and for the Hodge filtration
$Gr_F^pGr_l^W=F^pGr_l^W/F^{p+1}Gr_l^W$.

\begin{definition}\textnormal{The \emph{Hodge numbers} of $H$ are$$h^{p,q}(H)=\dim Gr_F^p
Gr_{p+q}^W.$$}
\end{definition}

\begin{parrafo}\textnormal{Let 
$H_{\mathbb{Q}}$ and $H_{\mathbb{Q}}'$ be two mixed Hodge structures, with filtrations $W_m$
and $F^p$, and $W_l'$ and ${F'}^q$ respectively. A \emph{morphism of mixed Hodge
structures} is a linear map $$L:H_{\mathbb{Q}} \rightarrow
H_{\mathbb{Q}}'$$ satisfying $L(W_m)\subset W'_{m}$ and
$L(F^p)\subset F'^{p}$. }\end{parrafo}

Deligne proved that the usual cohomology groups $H^k(X,\mathbb{Q})$ and those with compact support, which we denote by
$H^k_c(X)$, of a complex variety $X$ which may be singular
and not projective, carry a mixed
Hodge structure (see \cite{D2}, \cite{D3} and \cite{D4}). 
\begin{definition}\textnormal{For any complex algebraic variety $X$,
we define its \emph{Hodge--Deligne polynomial} (or virtual Hodge
polynomial) as (see \cite{DK})
$$\mathcal{H}(X)(u,v)=\sum_{p,q,k}(-1)^{p+q+k}h^{p,q}(H_c^k(X))u^p v^q \in \mathbb{Z}[u,v].$$We define its \emph{Hodge--Poincar\'{e} polynomial} as
$$HP(X)(u,v)=\sum_{p,q,k}(-1)^{p+q+k}h^{p,q}(H^k(X))u^p v^q .$$}
\end{definition}

Danilov and Khovanski\v{i} (\cite{DK}) observed that
$\mathcal{H}(X)(u,v)$ coincides with the classical Hodge
polynomial when $X$ is smooth and projective. 

\begin{remark}\textnormal{When our algebraic variety $X$ is smooth, Poincar\'{e} duality gives us the
following functional identity relating Hodge--Deligne and
Hodge--Poincar\'{e} polynomials
\begin{equation}\label{identityDP}\mathcal{H}(X)(u,v)=(uv)^{\dim_{\mathbb{C}}
X}\cdot HP(X)(u^{-1},v^{-1})
\end{equation}where $\dim_{\mathbb{C}}
X$ denotes the complex dimension of $X$.}\end{remark}

\begin{theorem}[\cite{MOV1}, Theorem 2.2]\label{Theorem2.2}Let $X$ be a complex
variety. Suppose that $X$ is a finite disjoint union $X=\cup_i
X_i$, where $X_i$ are locally closed algebraic subvarieties. Then
$$\mathcal{H}(X)(u,v)=\sum_i\mathcal{H}(X_i)(u,v).$$
\end{theorem}

Another result from \cite{MOV1} that will be useful for our
computations is
\begin{lemma}[\cite{MOV1}, Lemma 2.3]\label{lema2.3}Suppose that $\pi :X\rightarrow Y$
is an algebraic fibre bundle with fibre $F$ which is locally
trivial in the Zariski topology, then
$$\mathcal{H}(X)(u,v)=\mathcal{H}(F)(u,v)\cdot\mathcal{H}(Y)(u,v).$$
\end{lemma}

If $X$ is an algebraic variety acted on by a group $G$, let $BG$ be the classifying space of $G$ and
$EG \rightarrow BG$ the classifying bundle. We form the space $X\times_G EG$ which
is defined to be the quotient space of $X\times EG$ by the free action of $G$ on $X\times EG$ given by $g\cdot (x,e)=(g\cdot x, e\cdot g^{-1})$. Then, the
\emph{equivariant cohomology ring} of $X$ is 
$$H^{\ast}_G(X)=H^{\ast}(X\times_G EG).$$ Although $EG$ and $BG$
are not finite-dimensional manifolds, there are natural Hodge
structures on their cohomology. This is trivial in the case of
$EG$. Deligne proved that there is a pure Hodge structure on
$H^{\ast}(BG)$ and that $H^{p,q}(H^{\ast}(BG))=0$ for $p\neq q$
(see \cite[Theorem 9.1.1]{D4}). We may regard $EG$ and $BG$ as increasing
unions of finite-dimensional varieties $(EG)_m$ and $(BG)_m$ for
$m\geq 1$ such that $G$ acts freely on $(EG)_m$ with
$(EG)_m/G=(BG)_m$ and the inclusions of $(EG)_m$ and $(BG)_m$ in
$EG$ and $BG$ respectively induce isomorphisms of cohomology in
degree less than $m$ which preserve the Hodge structures. In the
same way $X\times_G EG$ is the union of finite-dimensional
varieties whose natural mixed Hodge structures induce a natural
mixed Hodge structure on $H^n (X\times_G EG)$. Using that we have
the following

\begin{definition}\textnormal{For any complex algebraic variety $X$ acted on by an algebraic group $G$,
we define its \emph{equivariant Hodge--Poincar\'{e} polynomial} as
$$HP_G(X)(u,v)=\sum_{p,q,k}(-1)^{p+q+k}h_G^{p,q;k}(X)u^p v^q ,$$where 
$h^{p,q;n}_G(X)=h^{p,q}(H^n (X\times_G
EG)).$}
\end{definition}

\begin{parrafo}\label{parraequicoho}\textnormal{Suppose now that $G$ is connected. The relationship
between cohomology and equivariant cohomology is accounted for by
a Leray spectral sequence for the
fibration\begin{equation}\label{FEC}X\times_G EG \rightarrow BG
\end{equation}whose fibre is $X$. The $E_2$-term of this spectral sequence is given
by $E_2^{p,q}=H^p(X)\otimes H^q(BG)$ which abuts to
$H^{p+q}_G(X)$. This spectral sequence preserves Hodge structures.}

\textnormal{If $X$ is a nonsingular projective variety that is acted on linearly by a connected complex reductive group $G$, one has that the fibration
(\ref{FEC}) is cohomologically trivial over $\mathbb{Q}$ (see
\cite[Proposition 5.8]{K1}). Then
\begin{equation}\label{cohoeq}H^{\ast}_G(X)\cong H^{\ast}(X)\otimes
H^{\ast}(BG).\end{equation}This isomorphism is actually an
isomorphism of mixed Hodge structures (\cite[Proposition 8.2.10]{D4}).}

\textnormal{We have another fibration, that is $$X\times_G EG
\rightarrow X/G$$with fibre $EG$. If $G$ acts freely on $X$,
then it induces
the isomorphism\begin{equation}\label{quotient}H^{\ast}(X\times_G
EG)\cong H^{\ast}(X/G).\end{equation}}\end{parrafo}

\begin{remark}\textnormal{Let
$GL(N)$ and $SL(N)$ be the general linear group and the special
linear group respectively. In this paper we only consider these
groups with complex coefficients. The Hodge--Poincar\'{e} series of
$BGL(N)$ and $BSL(N)$ are given by 
\begin{equation}\label{BSL}\textnormal{$HP(BGL(N))(u,v)=\prod_{1\leq k\leq
N}(1-u^k v^k)^{-1}$ $ $ $ $ and $ $ $ $ $HP(BSL(N))(u,v)=\prod_{2\leq k\leq
N}(1-u^k v^k)^{-1}$.}\end{equation}}
\end{remark}
\section{Stratifications}\label{Stratifications}

Let $X$ be a nonsingular complex
projective variety in $\mathbb{P}^n$ and $G$ a reductive group
that acts linearly on $X$. Assume that $X$ is embedded in
$\mathbb{P}^n$ by a line bundle $L$ which is the restriction of
the hyperplane bundle $H$ on $\mathbb{P}^n$ to $X$. We denote the set of semistable points by $X^{ss}$ and the good quotient or GIT quotient by $X^{ss}/G=X/\!/G$ together with the quotient map $X^{ss}\rightarrow X/\!/G$. Regarding the set of stable points, throughout this paper it will be convenient to use Mumford's original definition of properly stable points, nowadays called stable.  A point $x\in X$ is properly stable if $\dim \mathcal{O}(x)=\dim G$
and there exists an invariant homogeneous polynomial $f$ of degree
$\geqslant 1$ such that $f(x)\neq 0$ and the action of $G$ on
$X_f$ is closed. We denote by $X^s_{(0)}$ the set of properly stable points of $X$ and by $X^s_{(0)}/G$ the geometric quotient of $X$ by $G$. We also denote by $X^s_{(i)}$ the set of points of $X$ that satisfy the same properties as the properly stable points but with $\dim \mathcal{O}(x)=\dim G-i$.

In this section we review the definitions and explicit descriptions of a couple of stratifications introduced by Kirwan in \cite{K1} and \cite{K6} respectively. We talked about the first one in Paragraph \ref{Introduccion stratification}.
This requires the hypothesis that $X^{ss}=X_{(0)}^{s}$. This
stratification turns out to be equivariantly perfect, which
implies that one may obtain inductive formulae for computing the
Betti numbers and Poincar\'e polynomials of $X/\!/G$ when $G$ is
connected.

For the latter, we stratify the set $X^{ss}$ in such a way that
the set of properly stable points $X_{(0)}^s$ is an open stratum.
Unfortunately, this stratification is not equivariantly perfect,
so one does not expect to obtain nice formulae for computing the
Betti numbers and Poincar\'e polynomials of $X_{(0)}^s/G$ when $G$
is connected. But this stratification may be used to compute the
Hodge--Poincar\'e polynomials of the properly stable part as we
will do later on in this paper.


\subsection{The general
construction}\label{Stratificationgeneral}When $X^{ss}=X_{(0)}^s$
there exists a stratification $\{ \mathcal{S}_{\beta}:\beta \in
\mathcal{B}\}$ of $X$ by nonsingular $G$-invariant locally closed
subvarieties $\mathcal{S}_{\beta}$ satisfying properties (i), (ii)
and (iii) of Paragraph \ref{Introduccion stratification}. This
stratification can be defined either using symplectic geometry and
the moment map, or algebraically.

When $X$ is a nonsingular complex projective variety in
$\mathbb{P}^n$ and $G$ a reductive group that acts linearly on
$X$, this stratification is defined as follows (see \cite{K1} for
details). Assume that $G$ acts on $X$ via a rational
representation $\rho :G\rightarrow GL(n+1)$. It is well known that one can choose coordinates so that
$\rho$ restricts to an unitary representation $\rho_K
:K\rightarrow U(n+1)$ where $K$ is a maximal compact subgroup of $G$ whose complexification is $G$. Let $T$ be a maximal compact torus of $K$,
and let $\mathfrak{t}$ be its Lie algebra. The maximal torus $T$
acts on $X$ via a morphism $T\rightarrow U(n+1)$; after
conjugating this morphism by an element of $U(n+1)$, we may assume
that $T$ acts via $$t\mapsto \diag (\alpha_0(t) ,\ldots
,\alpha_n(t)),$$where $\alpha_j$ are characters of $T$. We
choose an inner product on $\mathfrak{t}$, invariant under the
action of the Weyl group, and use this to identify $\mathfrak{t}$
and its dual, $\mathfrak{t}^{\ast}$. Under this inner product, we
identify the characters $\alpha_j$ with points in
$\mathfrak{t}^{\ast}$, these are the weights of $T$. By abuse of
notation, we denote the weights by $\alpha_j$. Let $\mathcal{W}:=
\{ \alpha_0 ,\ldots \alpha_n \}$ be the set of weights for the action
of $T$. Then, the indexing set $\mathcal{B}$ is defined as
follows. An element $\beta \in \mathfrak{t}^{\ast}$ belongs to $ \mathcal{B}$ if and only if $\beta$ is the closest point to 0 of the convex hull, $\Conv S$, of some nonempty subset $S$ of $\mathcal{W}$. Then, if $\beta \in \mathcal{B}$, $\beta$
is the closest point
to 0 of the convex hull$$\Conv \{ \alpha_i \in
\mathcal{W}\textnormal{$ $ such that $ $} \alpha_i . \beta =\|
\beta \|^2\},$$where $.$ is the inner product and $\| \cdot \|$
its associated norm. If we choose a positive Weyl chamber of
$\mathfrak{t}^{\ast}$, the indexing set
$\mathcal{B}$ can be then identified with a finite set of points
in $\mathfrak{t}_+^{\ast}$.

Regarding the varieties $Z_{\beta}$ of Paragraph \ref{Introduccion
stratification} (iii), we define
\begin{equation}\label{ZetaBeta}Z_{\beta} :=\{ (x_0 : \ldots :x_n)\in X
\textnormal{ $ $ such that $x_i=0$ if $\alpha_i . \beta \neq
\parallel \beta \parallel^2$}\}\end{equation}and\begin{equation}\label{YBeta}Y_{\beta} :=\{ (x_0:\ldots
:x_n)\in X \textnormal{ $ $ : $x_i=0$ if $\alpha_i . \beta <
\parallel \beta \parallel^2$ and $\exists$ $x_i\neq 0$ with
$\alpha_i . \beta = \parallel \beta
\parallel^2$}\}.\end{equation}The variety $Z_{\beta}$ is a proper closed
subvariety of $X$ and $Y_{\beta}$ is a locally closed subvariety.
There is a retraction $p_{\beta}:Y_{\beta}\rightarrow Z_{\beta}$
defined by \begin{equation}\label{pbeta}p_{\beta}(x_0:\ldots
:x_n)=(x_0':\ldots :x_n')\end{equation}such that $x_i'=x_i$ if
$\alpha_i . \beta =\| \beta \|^2$ and $x_i'=0$
otherwise. Moreover, the subvarieties $Z_{\beta}$ and $Y_{\beta}$
are nonsingular and $p_{\beta}$ is a locally trivial fibration
whose fibre at any point is isomorphic to $\mathbb{C}^{m_{\beta}}$
for some $m_{\beta}\geqslant 0$ (see \cite[Corollary 13.2]{K1}).

Let $\Stab \beta$ be the stabiliser of $\beta$ under the adjoint
action of $G$. Let $Z_{\beta}^{ss}$ be the set of semistable
points of $Z_{\beta}$ with respect to the action of $\Stab \beta$
properly linearised (for details see \cite{K1}). Let
$Y_{\beta}^{ss}:=p_{\beta}^{-1}(Z_{\beta}^{ss})$. Then, the
restriction of $p_{\beta}$ to $Y_{\beta}^{ss}$
$$p_{\beta}:Y_{\beta}^{ss}\rightarrow Z_{\beta}^{ss}$$is a
locally trivial fibration whose fibre is $\mathbb{C}^{m_{\beta}}$
for some $m_{\beta}\geqslant 0$. Let $B$ be the Borel subgroup of
$G$ associated to the choice of positive Weyl chamber
$\mathfrak{t}_+$ and let $P_{\beta}\subseteq G$ consist of all $g \in G$ such that 
$\underset{t\rightarrow
\infty}{\lim} \exp (-it\beta)g\exp (it\beta)$ is an element of
$G$. In Lemma 6.9 of \cite{K1} it is proved that $P_{\beta}$ is the parabolic subgroup
$B\Stab \beta$. Moreover, $Y_{\beta}$
and $Y_{\beta}^{ss}$ are $P_{\beta}$-invariant and
\begin{equation}\label{parabolico}\mathcal{S}_{\beta}\cong
G\times_{P_\beta}Y_{\beta}^{ss}\end{equation}(see \cite{K1}). The limit in the definition of $P_{\beta}$ defines a surjection $q_{\beta}:P_{\beta}
\rightarrow \Stab \beta$ which is actually a retraction. From (\ref{parabolico}) and the
fact that $p_{\beta}$ and $q_{\beta}$ are retractions, one deduces
that $$H_G^{\ast}(\mathcal{S}_{\beta})\cong H_{\Stab
\beta}^{\ast}(Z_{\beta}^{ss})$$which is property (iii) of
Paragraph \ref{Introduccion stratification}. Moreover, the
stratification $\{ \mathcal{S}_{\beta} \}_{\beta \in \mathcal{B}}$
satisfies that $\mathcal{S}_0=X^{ss}$ and is equivariantly perfect
(see properties (i) and (ii) of Paragraph \ref{Introduccion
stratification}). One has that $X=\bigcup_{\beta \in
\mathcal{B}}\mathcal{S}_{\beta }$, and there is a partial order on
$\mathcal{B}$ such that\begin{equation}\label{partial
order}\overline{S_{\beta }}\subseteq S_{\beta }\cup
\bigcup_{\gamma
> \beta}S_{\gamma},\end{equation}where $\gamma
> \beta$ if $\| \gamma \|> \| \beta \|$.

\subsection{A stratification for the set of semistable points}\label{StratificationSection2}
When $X^{ss}\neq X^{s}_{(0)}\neq \emptyset$ the GIT quotient
$X/\!/G$ may have serious singularities. For the moduli space of
semistable vector bundles, $\mathcal{M}(n,d)$, the problem of
finding natural desingularisations has been studied by Seshadri
\cite{S}, Narasimhan and Ramanan \cite{NR}, and Kirwan
\cite{K4}. The latter is an application of the method described in
\cite{K2} that works for a smooth complex projective variety $X$
acted on by a reductive group $G$ and such that $X^{ss}\neq
X^{s}_{(0)}\neq \emptyset$. From her method, Kirwan defines a refined stratification of $X^{ss}$ such that
$X^{s}_{(0)}$ is an open stratum (see \cite{K6}). This method consists of blowing
up $X$ along a sequence of smooth $G$-invariant subvarieties to
obtain a $G$-invariant morphism $\pi :
\widetilde{X}^{ss}\rightarrow X^{ss}$, where $ \widetilde{X}^{ss}
$ is a projective variety acted on linearly by $G$ properly
lifted, and such that $\widetilde{X}^{ss}=\widetilde{X}^{s}_{(0)}$
with respect to the induced action. The induced birational
morphism $\widetilde{X}/\!/G \rightarrow X/\!/G$ can be regarded as a
``partial desingularisation'' of the GIT quotient $X/\!/G$ in the
sense that the more serious singularities of $X/\!/G$ have been
resolved. The only singularities of $\widetilde{X}/\!/G$ are finite
quotient singularities (for more details see \cite{K2}).

\begin{parrafo}\textnormal{This desingularisation process is based on the fact that
there exist semistable points that are not properly stable if and
only if there exists a non-trivial connected reductive subgroup
of $G$ fixing a semistable point. Let $r_1>0$ be the maximal
dimension of a reductive subgroup of $G$ fixing a point of
$X^{ss}$ and let $\mathcal{R}(r_1)$ be a set of representatives of
conjugacy classes of all connected reductive subgroups $R_1$ of
dimension $r_1$ in $G$ such that $$Z_{R_1}^{ss}:=\{ x\in X^{ss}
\textnormal{$ $ $ $ such that $ $ $R_1$ fixes $x$} \}$$is
non-empty. We consider $$\cup_{R_1 \in
\mathcal{R}(r_1)}GZ_{R_1}^{ss}$$where $GZ_{R_1}^{ss}:= \{ gx$ such
that $g\in G$ and $x\in Z_{R_1}^{ss}\}$. These sets are
non-singular closed subvarieties of $X^{ss}$. In the first step
one blows up $X^{ss}$ along the subvariety $\cup_{R_1 \in
\mathcal{R}(r_1)}GZ_{R_1}^{ss}$. In \cite[Corollary 8.3]{K2} it is proved
that the blow-up of $X^{ss}$ along the subvariety $\cup_{R_1 \in
\mathcal{R}(r_1)}GZ_{R_1}^{ss}$ is the same as the result of
blowing up $X^{ss}$ along each $GZ_{R_1}^{ss}$ for $R_1 \in
\mathcal{R}(r_1)$ one at a time. Let $X_{(R_1)}$ be the blown-up
variety along $\cup_{R_1 \in \mathcal{R}(r_1)}GZ_{R_1}^{ss}$ and
$E_1$ be the exceptional divisor. The action of $G$ on $X^{ss}$
lifts to an action on $X_{(R_1)}$ with respect to
$\pi_1^{\ast}L^{\otimes k}\otimes \mathcal{O}(-E_1)$ where $\pi_1
:X_{(R_1)}\rightarrow X^{ss}$ and $k$ is any integer. When $k$ is
large enough the set $X_{(R_1)}^{ss}$ with respect to the lifted
action is independent of $k$.}\end{parrafo}

\begin{parrafo}\label{Par de condiciones}\textnormal{In \cite[Remark 7.17]{K2},
the set $X_{(R_1)}^{ss}$ is characterized by the following
property:\begin{itemize}\item[(a)] The complement of
$X_{(R_1)}^{ss}$ in $X_{(R_1)}$, i.e. $X_{(R_1)}\backslash
X_{(R_1)}^{ss}$, is the proper transform of the subset
$\phi^{-1}(\phi (GZ_{R_1}^{ss}))$ of $X^{ss}$ where $\phi :X^{ss}
\rightarrow X/\!/G$ is the quotient map.\end{itemize}It is also satisfied that:\begin{itemize}\item[(b)] No point of
$X_{(R_1)}^{ss}$ is fixed by a reductive subgroup of $G$ of
dimension at least $r_1$, and a point in
$E_1^{ss}=X_{(R_1)}^{ss}\cap E_1$ is fixed by a reductive subgroup $R'$
of dimension $r' < r_1$ in $G$ if and only if it belongs to
the proper transform of the subvariety $Z_{R'}^{ss}$ in
$X^{ss}$. We denote the proper transform of $Z_{R'}^{ss}$ of
$X^{ss}$ by $\widehat{Z}_{R'}^{ss}$.\end{itemize}}\end{parrafo}

If one repeats this process for $X_{(R_1)}^{ss}$ and so on, after at
most $r_1-1$ steps one obtains a $G$-invariant morphism $\pi :
\widetilde{X}^{ss}\rightarrow X^{ss}$, where $ \widetilde{X}^{ss}
$ is a projective variety acted on linearly by $G$ properly
lifted, and such that
$\widetilde{X}^{ss}=\widetilde{X}^{s}_{(0)}$. This is equivalent
to constructing a sequence of varieties
$$X^{ss}_{(R_0)}=X^{ss}, X^{ss}_{(R_1)}, \ldots ,
X^{ss}_{(R_{\tau})}=\widetilde{X}^{ss}$$where $R_1$, $\ldots$,
$R_{\tau}$ are connected reductive subgroups of $G$ with
$$r_1=\dim R_1 \geqslant \dim R_2 \geqslant \ldots \geqslant \dim
R_{\tau} \geqslant 1,$$and if $1\leqslant l \leqslant \tau $ then
$X_{(R_l)}$ is the blow-up of $X^{ss}_{(R_{l-1})}$ along its
closed nonsingular subvarieties $G\widehat{Z}^{ss}_{R_l}$. We have $G\widehat{Z}^{ss}_{R_l} \cong G\times_{N_l}\widehat{Z}^{ss}_{R_l}$, where $N_l$
is the normaliser of $R_l$ in $G$. Similarly,
$\widetilde{X}/\!/G=\widetilde{X}^{ss}/G$ can be obtained from
${X}/\!/G$ by blowing up along the proper transforms of the subvarieties $\phi (GZ^{ss}_R)=Z_R/\!/(N/R)$ of $X/\!/G$ in decreasing order of $\dim R$, where $\phi :X^{ss}
\rightarrow X/\!/G$. Note that
\begin{equation}\label{GIT para las variedades}GZ^{ss}_R/G\cong
Z_R/\!/(N/R).\end{equation}

As in Subsection \ref{Stratificationgeneral}, for each $1\leqslant
l \leqslant \tau $ one has a $G$-equivariant stratification$$\{
\mathcal{S}_{\beta ,l } :(\beta ,l)\in \mathcal{B}_l\times \{ l \}
\}$$of $X_{R_l}$ by nonsingular $G$-invariant locally closed
subvarieties such that one of the strata, indexed by $(0,l)\in
\mathcal{B}_l \times \{ l \}$, coincides with the open subset
$X^{ss}_{(R_l)}$ of $X_{(R_l)}$. There is a partial ordering in
$\mathcal{B}_l$ given by $\gamma > \beta $ if
$\parallel\gamma\parallel >
\parallel \beta \parallel$. Then $0$ is its minimal element, and
if $\beta \in \mathcal{B}_l$ then the closure in $X_l$ of the
stratum $\mathcal{S}_{\beta ,l}$
satisfies$$\overline{\mathcal{S}_{\beta ,l}}\subseteq
\bigcup_{\gamma \in \mathcal{B}_l , \gamma \geqslant
\beta}\mathcal{S}_{\gamma ,l}.$$If $\beta \in \mathcal{B}_l$ and
$\beta \neq 0$ then the stratum $\mathcal{S}_{\beta ,l}$ retracts
$G$-equivariantly onto its transverse intersection with the
exceptional divisor $E_l$ for the blow-up $X_{(R_l)}\rightarrow
X^{ss}_{(R_{l-1})}$. This exceptional divisor is isomorphic to the
projective bundle $\mathbb{P}(\mathcal{N}_l)$ over
$G\widehat{Z}^{ss}_{R_l}$, where $\mathcal{N}_l$ is the normal bundle to $G\widehat{Z}^{ss}_{R_l}$
in $X^{ss}_{(R_{l-1})}$.

\begin{parrafo}\label{stratification fibra}\textnormal{Now, let
$$\pi_l :E_l\rightarrow
G\widehat{Z}^{ss}_{R_l}$$be the projection obtained from the
restriction of $\pi_l :X_{(R_l)}\rightarrow X_{(R_{l-1})}^{ss}$ to
$E_l$. This restriction is a locally trivial
fibration whose fibre is isomorphic to
$\mathbb{P}(\mathcal{N}_l)$. It is proved in \cite{K2} that the
stratification $\{ \mathcal{S}_{\beta ,l }:\beta \in \mathcal{B}_l
\}$ is determined by the action of $R_l$ on the fibres of
$\mathcal{N}_l$ over $G\widehat{Z}^{ss}_{R_l}$. More precisely, in
Lemma 7.9 of \cite{K2} it is shown that when $x\in
\widehat{Z}^{ss}_{R_l}$ the intersection of $\mathcal{S}_{\beta
,l}$ with the fibre $\pi_l^{-1}(x) =\mathbb{P}(\mathcal{N}_{l,x})$
of $\pi_l$ at $x$ is the union of those strata indexed by points
in the adjoint orbit Ad$(G)\beta$ in the stratification of
$\mathbb{P}(\mathcal{N}_{l,x})$ induced by the representation
$\rho_l$ of $R_l$ on the normal $\mathcal{N}_{l,x}$ to
$G\widehat{Z}^{ss}_{R_l}$ at $x$. Let $\mathcal{B}(\rho_l)$ be the
indexing set corresponding to $\rho_l$; if each Ad$(G)$-orbit
meets $\mathcal{B}(\rho_l)$ in at most one point then one may
assume that
$$\mathcal{B}_l=\mathcal{B}(\rho_l).$$Moreover, each stratum
$\mathcal{S}_{\beta ,l}$ retracts onto its intersection with the
exceptional divisor $E_l$, and if
$\mathcal{B}_l=\mathcal{B}(\rho_l)$ then this intersection
retracts onto \begin{equation*}G\times_{N_l \cap \Stab
\beta}Z_{\beta ,l}^{ss}\cap
\pi_l^{-1}(\widehat{Z}_{R_l}^{ss}),\end{equation*}where
\begin{equation}\label{Fibration restrition}\pi_l : Z_{\beta ,l}^{ss}\cap
\pi_l^{-1}(\widehat{Z}_{R_l}^{ss})\rightarrow
\widehat{Z}_{R_l}^{ss},\end{equation}is a fibration with fibre
${Z}_{\beta ,l}^{ss}(\rho_l)$. The variety
${Z}_{\beta ,l}^{ss}(\rho_l)$ is defined as ${Z}_{\beta }^{ss}$
but with respect to the induced action by $\rho_l$ of $R_l$ on
$\mathbb{P}(\mathcal{N}_l)$. If $\beta$ is a maximal element of
$\mathcal{B}(\rho_l)$ with respect to the partial order, then
${Z}_{\beta ,l}^{ss}(\rho_l)={Z}_{\beta ,l}(\rho_l)$ which is a
projective space. By \cite[Lemma 7.11]{K2} and
\cite[Theorem 6.18]{K1}, the codimension of $\mathcal{S}_{\beta ,l}$ in
$X_{(R_l)}$ is given by\begin{equation}\label{codimension de los
estratos normales}d(\beta ,l):=\codim \mathcal{S}_{\beta
,l}=n(\beta ,l)-\dim R_l /B\Stab \beta,\end{equation}where
$n(\beta ,l)$ is the number of weights $\alpha$ of the
representation $\rho_l$ such that $\alpha . \beta < \parallel
\beta
\parallel^2$.}
\end{parrafo}

Then, there is a stratification $\{ \Sigma_{\gamma} \}_{\gamma}$
of $X^{ss}$ (see \cite{K6}) indexed by
\begin{equation}\label{gamma}\Gamma =\{ R_1 \} \sqcup \{ R_1 \}\times
\{ \mathcal{B}_1 \backslash \{ 0 \} \} \sqcup \ldots \sqcup \{
R_{\tau} \} \sqcup \{ R_{\tau} \}\times \{ \mathcal{B}_{\tau}
\backslash \{ 0 \} \} \sqcup \{ 0 \} \end{equation}defined as
follows. One takes as the highest stratum $\Sigma_{R_1}$ the
nonsingular closed subvariety $GZ^{ss}_{R_1}$ whose complement in
$X^{ss}$ can be naturally identified with $X_{(R_1)}\backslash
E_1$. One has $GZ^{ss}_{R_1} \cong G\times_{N_1}Z^{ss}_{R_1}$
where $N_1$ is the normaliser of $R_1$ in $G$, and $Z^{ss}_{R_1}$
is equal to the set of semi-stable points for the action of
$N_1$, or equivalently for the induced action of $N_1/R_1$, on
$Z_{R_1}$, which is a union of connected components of the fixed
point set of $R_1$ in $X$. Moreover since $R_1$ has maximal
dimension among those reductive groups of $G$ with fixed points in
$X^{ss}$, we have $Z^{ss}_{R_1}=Z^{s}_{R_1}$ where $Z^s_{R_l}$
denotes, for each $l$, the set of properly stable points for the
action of $N_l/R_l$ on $Z_{R_l}$ for $1\leqslant l\leqslant \tau$.

\begin{remark}\textnormal{Note that $x\in Z_{R_1}$ is properly
stable for the action of $N_1/R_1$ if and only if $x\in
(Z_{R_1})^s_{(r_1)}$ for the action of $N_1$ on $Z_{R_1}$ where
$r_1=\dim R_1$. Since $GZ^{ss}_{R_1}\cong
G\times_{N_1}Z^{ss}_{R_1}$ and is closed in $X^{ss}$, one also has
that $x\in Z_{R_1}$ is properly stable for the action of $N_1/R_1$
if and only if $x\in (\overline{GZ_{R_1}})^s_{(r_1)}$ for the
action of $G$ on $\overline{GZ_{R_1}}$.}\end{remark}

We take as our next strata the nonsingular locally closed
subvarieties $$\Sigma_{\beta ,1}=\mathcal{S}_{\beta ,1} \backslash
E_1,$$for $\beta \in \mathcal{B}_1$ with $\beta \neq 0$, of
$X_{(R_1)}\backslash E_1=X^{ss}\backslash GZ^{ss}_{R_1}$, whose
complement in $X_{(R_1)}\backslash E_1 $ is just
$X^{ss}_{(R_1)}\backslash E_1 =X^{ss}_{(R_1)}\backslash E^{ss}_1$
where $E^{ss}_1 =X^{ss}_{(R_1)}\cap E_1$, and then we take the intersection of $X^{ss}_{(R_1)}\setminus E_1^{ss}\subseteq X^{ss}\setminus GZ_{R_1}^{ss}$ with $GZ^{ss}_{R_2}\subseteq X^{ss}$. 

\begin{remark}\textnormal{In \cite{K5} and \cite{K6} it is claimed that this intersection is $GZ_{R_2}^s$ where $Z^s_{R_2}$ is the set of properly stable points for the action of $N_2/R_2$ on $Z_{R_2}$. This is not necessarily true; in \cite{K7} it is explained that instead one needs to consider $Z^s_{R_2}$ to be the
intersection of $Z_{R_2}$ with $(\overline{GZ_{R_2}})^s_{(r_2)}$
defined for the action of $G$ on $\overline{GZ_{R_2}}$ where $r_2
=\dim R_2$. This is an open subset of the set of properly stable
points for the action of $N_2/R_2$ on $Z_{R_2}$, which is the
intersection of $Z_{R_2}$ with $(\overline{GZ_{R_2}})^s_{(r_2)}$
defined for the action of $N_2$ (not $G$) on $\overline{GZ_{R_2}}$.}\end{remark}

The next strata are the nonsingular locally closed subvarieties
$$\Sigma_{\beta ,2}=\mathcal{S}_{\beta ,2} \backslash
(E_2 \cup \widehat{E}_1),$$for $\beta \in \mathcal{B}_2$ with
$\beta \neq 0$, of $X_{(R_2)}\backslash (E_2 \cup \widehat{E}_1)$,
whose complement in $X_{(R_1)}\backslash (E_2 \cup \widehat{E}_1)
$ is $X^{ss}_{(R_1)}\backslash (E_2 \cup \widehat{E}_1)$. The
stratum after these is $GZ^s_{R_3}$, where $Z^s_{R_3}$ is the
intersection of $Z_{R_3}$ with $(\overline{GZ_{R_3}})^s_{(r_3)}$
defined for the action of $G$ on $\overline{GZ_{R_3}}$ where $r_3
=\dim R_3$.

Then, in general one has two different types of strata. For each
$1\leqslant l\leqslant \tau$,
either$$\Sigma_{R_l}=GZ^s_{R_l},$$where $Z^s_{R_l}$ is the
intersection of $Z_{R_l}$ with $(\overline{GZ_{R_l}})^s_{(r_l)}$
defined for the action of $G$ on $\overline{GZ_{R_l}}$ where $r_l
=\dim R_l$, or$$\Sigma_{\beta ,l}=\mathcal{S}_{\beta ,l}
\backslash (E_l \cup \widehat{E}_{l-1}\cup \ldots \cup
\widehat{E}_1),$$for $\beta \in \mathcal{B}_l$ with $\beta \neq
0$, of $X_{(R_l)}\backslash (E_l \cup \widehat{E}_{l-1}\cup \ldots
\cup \widehat{E}_1)$, whose complement in $X_{(R_l)}\backslash
(E_l \cup \widehat{E}_{l-1}\cup \ldots \cup \widehat{E}_1) $ is
$X^{ss}_{(R_l)}\backslash (E_l \cup \widehat{E}_{l-1}\cup \ldots
\cup \widehat{E}_1)$. We take $\Sigma_0 =X^{s}_{(0)}$ as our final
stratum, which is the unique one indexed by $0$. This is actually
the unique open stratum for the stratification.

One has a partial order induced on $\Gamma$, given by the
partial orders of $\mathcal{B}_i$ for all $i$ together with the
ordering in the expression (\ref{gamma}). For this partial
ordering, the maximal element is $R_1$ and the minimal element is
$0$, and the closure of each stratum $\Sigma_{\gamma}$ indexed by
$\gamma \in \Gamma$ in $X^{ss}$,
satisfies$$\overline{\Sigma_{\gamma}}\subseteq
\bigcup_{\tilde{\gamma}\geqslant \gamma}
\Sigma_{\tilde{\gamma}}.$$


\subsubsection{Explicit description}


In this subsection we give an explicit description of the previous strata
(see \cite{K6} for details). There are two different
types of strata. For each $1\leqslant l\leqslant \tau$,
either$$\Sigma_{R_l}=GZ^s_{R_l},$$in this case $GZ^s_{R_l}\cong
G\times_{N_l}Z^s_{R_l}$, or$$\Sigma_{\beta ,l}=\mathcal{S}_{\beta
,l} \backslash (E_l \cup \widehat{E}_{l-1}\cup \ldots \cup
\widehat{E}_1),$$for $\beta \in \mathcal{B}_l$ with $\beta \neq
0$. For the latter, recall from Subsection
\ref{Stratificationgeneral} that
$\mathcal{S}_{\beta ,l} =GY^{ss}_{\beta ,l} \cong
G\times_{P_{\beta}}Y^{ss}_{\beta ,l}$ where $Y^{ss}_{\beta ,l}$
fibres over $Z^{ss}_{\beta ,l}$ with fibre $\mathbb{C}^{m_{\beta
,l}}$ for some $m_{\beta ,l}>0$, and $P_{\beta}$ is the parabolic
subgroup BStab$(\beta)$ of $G$. By Lemmas 7.6 and 7.11 of
\cite{K2} $$\mathcal{S}_{\beta ,l} \cap E_l = G(Y^{ss}_{\beta
,l}\cap E_l) \cong G \times_{P_{\beta }}(Y^{ss}_{\beta ,l}\cap
E_l)$$where $Y^{ss}_{\beta ,l}\cap E_l$ fibres over $Z^{ss}_{\beta
,l}$ with fibre $\mathbb{C}^{m_{\beta ,l}-1}$. Thus
$$\mathcal{S}_{\beta ,l} \backslash E_l
\cong G \times_{P_{\beta }}(Y^{ss}_{\beta ,l}\backslash
E_l)$$where $Y^{ss}_{\beta ,l}\backslash E_l$ fibres over
$Z^{ss}_{\beta ,l}$ with fibre $\mathbb{C}^{m_{\beta ,l}-1}\times
(\mathbb{C}\backslash \{ 0 \})$.

In \cite{K6} it is proved that one can replace the indexing set
$\mathcal{B}_l\backslash \{ 0 \}$, whose elements correspond to
the $G$-adjoint orbits Ad$(G)\beta$ of elements of the indexing
set for the stratification of $\mathbb{P}(\mathcal{N}_{l,x})$
induced by the representation $\rho_l$ (see Paragraph
\ref{stratification fibra}), by the set of $N_l$-adjoint orbits
Ad$(N_l)\beta$. If $q_{\beta}:P_{\beta}\rightarrow
\textnormal{Stab}(\beta)$ is the projection, then
\begin{equation}\label{15} \Sigma_{\beta ,l}=GY_{\beta
-l}^{\setminus E}=G\times_{P_{\beta}}Y_{\beta -l}^{\setminus
E}\end{equation}where
\begin{equation}\label{fibrationnose}Y_{\beta
-l}^{\setminus E}=Y_{\beta ,l}^{ss}\backslash (E_l \cup
\widehat{E}_{l-1}\cup \ldots \cup \widehat{E}_{1})\rightarrow
\Stab \beta \times_{N_l \cap \Stab \beta } (Z_{\beta ,l}^{ss}\cap
\pi_l^{-1}(Z_{R_l}^s))\end{equation}is a fibration with fibre
$\mathbb{C}^{m_{\beta ,l}-1}\times (\mathbb{C} \setminus \{ 0 \})$,
and\begin{equation}\label{fibrationnose2}Z_{\beta ,l}^{ss}\cap
\pi_l^{-1}(Z_{R_l}^s)\rightarrow Z_{R_l}^s\end{equation}is a
fibration with fibre $Z_{\beta}^{ss}(\rho_l)$. We set $$Y_{\beta
,l}^{\backslash E}=Y_{\beta ,l}^{ss}\backslash (E_l \cup
\widehat{E}_{l-1}\cup \ldots \cup \widehat{E}_{1}) \cap
p_{\beta}^{-1}(Z_{\beta ,l}^{ss}\cap
\pi_{l}^{-1}({Z}_{R_l}^{s})).$$ It is proved in \cite{K6} that
$Y_{\beta ,l}^{ss} \cong P_{\beta} \times_{Q_{\beta,l}} Y_{\beta
-l}^{\setminus E}$, where $Q_{\beta ,l}=q_{\beta}^{-1}(N_l \cap
\Stab \beta )$. Then
\begin{equation}\label{stratum-nose} \Sigma_{\gamma}=\Sigma_{\beta
,l} \cong G\times_{Q_{\beta ,l}}Y_{\beta ,l}^{\backslash
E}\end{equation}where
\begin{equation}\label{eq1,1}p_{\beta}:Y_{\beta ,l}^{\backslash E}
\rightarrow Z_{\beta ,l}^{ss}\cap
\pi_{l}^{-1}({Z}_{R_l}^{s})\end{equation}is a fibration with fibre
$\mathbb{C}^{m_{\beta ,l}-1}\times (\mathbb{C}\backslash \{ 0 \}
)$ for a certain $m_{\beta ,l}>0$.


\section{Cohomological formulae}\label{Cohomological formulae}


Using the previous stratifications introduced by Kirwan, in this section we study the equivariant Hodge--Poincar\'e series
of the set of properly stable points of a complex projective
variety $X$ equipped with a linear action of a complex reductive
group $G$ and such that $X_{(0)}^{s}$ is nonempty and there are
semistable points in $X^{ss}$ that are not properly stable. When
$G$ is connected these formulae allow us to compute the
Hodge--Poincar\'e series of the geometric quotient
$X_{(0)}^{s}/G$.

To do that, we consider the stratification $\{ \Sigma_{\gamma}
\}_{\gamma}$ of $X^{ss}$ indexed by
$$\Gamma =\{ R_1 \} \sqcup \{ R_1 \}\times \{ \mathcal{B}_1
\backslash \{ 0 \} \} \sqcup \ldots \sqcup \{ R_{\tau} \} \sqcup
\{ R_{\tau} \}\times \{ \mathcal{B}_{\tau} \backslash \{ 0 \} \}
\sqcup \{ 0 \} ,$$where $\Sigma_0 =X_{(0)}^{s}$ is an open stratum,
of Subsection \ref{StratificationSection2}.
Since $X_{(0)}^{s}$ is an open set of $X^{ss}$, both have the same
dimension. We have the following identity for the equivariant
Hodge-Poincar\'{e} series of $X_{(0)}^{s}$.
\begin{proposition}\label{propositionidentity}With the previous notation, we have that
\begin{equation}HP_G(X_{(0)}^{s})(u,v)=HP_G(X^{ss})(u,v)-\sum_{\gamma
\in \Gamma \backslash \{ 0 \}}(uv)^{\lambda(\gamma)}HP_G(\Sigma_{\gamma})(u,v),
\end{equation}where $\lambda(\gamma)$ is the complex codimension
of $\Sigma_{\gamma}$ in $X^{ss}$.
\end{proposition}
\begin{proof}For each $\gamma \in \Gamma$ regard $EG$
as an increasing union of smooth finite-dimensional varieties
$(EG)_m$ where $G$ acts freely on $(EG)_m$. Let
$$\Sigma_{\gamma}\times_{G}EG =\cup_{m\geqslant
0}(\Sigma_{\gamma}\times_{G}EG)_m$$where
$(\Sigma_{\gamma}\times_{G}EG)_m=\Sigma_{\gamma}\times_{G}(EG)_m$,
and such that the immersion
$(\Sigma_{\gamma}\times_{G}EG)_m\hookrightarrow
\Sigma_{\gamma}\times_{G}EG$ induces isomorphisms in cohomology in
degrees less than or equal to $m$. From this decomposition one
obtains the following
decomposition\begin{align*}X^{ss}\times_{G}EG&=\sqcup_{\gamma \in
\Gamma}\Sigma_{\gamma}\times_{G}EG =\sqcup_{\gamma \in
\Gamma}\cup_{m\geqslant 0}(\Sigma_{\gamma}\times_{G}EG)_m=\\& =
\cup_{m\geqslant 0}\sqcup_{\gamma \in
\Gamma}(\Sigma_{\gamma}\times_{G}EG)_m=\cup_{m\geqslant
0}(X^{ss}\times_{G}EG)_m.\end{align*}Then, for each $m$ we have
that $(X^{ss}\times_{G}EG)_m =\sqcup_{\gamma \in
\Gamma}(\Sigma_{\gamma}\times_{G}EG)_m$ where, because of our
choices, $(X^{ss}\times_{G}EG)_m$ and
$(\Sigma_{\gamma}\times_{G}EG)_m$ are smooth finite-dimensional
varieties for every $m$ and $\gamma$. Using identity
(\ref{identityDP}) we obtain
$$HP((X^{ss}\times_{G}EG)_m)(u,v)=(uv)^{\dim_{\mathbb{C}}(X^{ss}\times_{G}EG)_m}\mathcal{H}
((X^{ss}\times_{G}EG)_m)(u^{-1},v^{-1}).$$Applying first Theorem
\ref{Theorem2.2} and then (\ref{identityDP}) we
get$$HP((X^{ss}\times_{G}EG)_m)(u,v)=\sum_{\gamma \in
\Gamma}(uv)^{\dim_{\mathbb{C}}(X^{ss}\times_{G}EG)_m
-\dim_{\mathbb{C}}(\Sigma_{\gamma}\times_{G}EG)_m}HP
((\Sigma_{\gamma}\times_{G}EG)_m)(u,v).$$Note that
$\dim_{\mathbb{C}}(X^{ss}\times_{G}EG)_m
-\dim_{\mathbb{C}}(\Sigma_{\gamma}\times_{G}EG)_m$ equals
$\dim_{\mathbb{C}}X^{ss}+\dim_{\mathbb{C}}(EG)_m
-\dim_{\mathbb{C}}G
-\dim_{\mathbb{C}}\Sigma_{\gamma}-\dim_{\mathbb{C}}(EG)_m
+\dim_{\mathbb{C}}G$, which is independent of $m$ and is actually
the codimension of $\Sigma_{\gamma}$ in $X^{ss}$. We name it
$\lambda(\gamma)$.

Since the immersion $(X^{ss}\times_{G}EG)_m$ in
$X^{ss}\times_{G}EG$ induces isomorphisms in cohomology in degrees
less than or equal to $m$, and the same is true for
$\Sigma_{\gamma}$, from the previous identity we have the
following
\begin{align*}\sum_{j=1}^m
\sum_{p,q}(-1)^{p+q+j}&h_G^{p,q;j}(X^{ss})u^pv^q+\sum_{j> m}
\sum_{p,q}(-1)^{p+q+j}h^{p,q}(H^j((X^{ss}\times_G
EG)_m))u^pv^q=\\& =\sum_{\gamma \in \Gamma}\sum_{j=1}^m
\sum_{p,q}(-1)^{p+q+j}h_G^{p,q;j}(\Sigma_{\gamma})u^{p+\lambda(\gamma)}v^{q+\lambda(\gamma)}+\\&
+ \sum_{\gamma \in \Gamma}\sum_{j> m}
\sum_{p,q}(-1)^{p+q+j}h^{p,q}(H^j((\Sigma_{\gamma}\times_G
EG)_m))u^{p+\lambda(\gamma)}v^{q+\lambda(\gamma)}.\end{align*}Now,
since this identity is independent of $m$ and $X_{(0)}^{s}$ and
$X^{ss}$ have the same dimension, we conclude.
\end{proof}

We shall need the following lemma for future computations.
\begin{lemma}\label{lemmafibrationEHP}Let $Y\rightarrow Z$ be
a locally trivial fibration in the Zariski topology with fibre
$F$, and such that it is compatible with respect to the action of
the group $G$ that acts on $Y$ and $Z$ respectively. Assume that $Y$ and $Z$ are smooth varieties. Then
$$HP_G(Y)(u,v)=HP_G(Z)(u,v)\cdot HP(F)(u,v).$$
\end{lemma}
\begin{proof}Regarding $Y\times_G EG$ and $Z\times_G EG$ as increasing
unions of smooth finite-dimensional varieties, for each $m$ we
have that the fibration $Y\rightarrow Z$ induces a new fibration
$Y\times_G (EG)_m\rightarrow Z\times_G (EG)_m$ with fibre $F$.
Now, applying (\ref{identityDP}) and Lemma \ref{lema2.3}, and
bearing in mind that the dimension of $Y$ is equal to the sum of
the dimension of $Z$ and the dimension of $F$, we obtain
$$HP(Y\times_G (EG)_m)(u,v)=HP(Z\times_G (EG)_m)(u,v)\cdot
HP(F)(u,v).$$Since this identity is independent of $m$, we finish
the proof of the lemma.\end{proof}

Following Subsection \ref{StratificationSection2} we have that the
strata $\Sigma_{\gamma}$ for $\gamma \neq 0$ fall into two
classes. Either $\gamma =R_l$ for some $l\in \{ 1, \ldots ,\tau
\}$, in which case
$$\Sigma_{R_l}=GZ_{R_l}^s$$or else $\gamma =(R_l ,\beta )$ where
$\beta \in \mathcal{B}_l \setminus \{ 0 \}$ for some $l\in \{ 1,
\ldots ,\tau \}$ and the stratum $\Sigma_{\gamma}$ is
$$\Sigma_{\gamma}=\Sigma_{\beta ,l}=S_{\beta ,l}\backslash (E_l
\cup \widehat{E}_{l-1}\cup \ldots \cup \widehat{E}_{1}).$$In the
first case, $\Sigma_{R_l}=GZ_{R_l}^s\cong G\times_{N_l} Z_{R_l}^s
$ where $N_l$ is the normaliser of $R_l$. Then
$H_G^{\ast}(\Sigma_{R_l})\cong H_{N_l}^{\ast}(Z_{R_l}^s)$ which is
an isomorphism of mixed Hodge structures, hence induces the
following identity
\begin{equation}\label{EHPsimplestratum}HP_G(\Sigma_{R_l})(u,v)= HP_{N_l}(Z_{R_l}^s)(u,v).
\end{equation}In the second case, we recall (see (\ref{15}), (\ref{fibrationnose}) and
(\ref{fibrationnose2})) that $\Sigma_{\beta
,l}=GY_{\beta -l}^{\setminus E}=G\times_{P_{\beta}}Y_{\beta
-l}^{\setminus E}$, where
\begin{equation}\label{fibrationnose3}Y_{\beta
-l}^{\setminus E}=Y_{\beta ,l}^{ss}\backslash (E_l \cup
\widehat{E}_{l-1}\cup \ldots \cup \widehat{E}_{1})\rightarrow
\Stab \beta \times_{N_l \cap \Stab \beta } (Z_{\beta ,l}^{ss}\cap
\pi_l^{-1}(Z_{R_l}^s))\end{equation}is a fibration with fibre
$\mathbb{C}^{m_{\beta ,l}-1}\times (\mathbb{C} \setminus \{ 0 \})$
for some $m_{\beta ,l}>0$,
and\begin{equation}\label{fibrationnose4}Z_{\beta ,l}^{ss}\cap
\pi_l^{-1}(Z_{R_l}^s)\rightarrow Z_{R_l}^s\end{equation}is a
fibration with fibre $Z_{\beta}^{ss}(\rho_l)$. From
(\ref{fibrationnose3}) we get the following fibration
\begin{equation*}G\times_{\Stab \beta } Y_{\beta
-l}^{\setminus E}\rightarrow G \times_{N_l \cap \Stab \beta }
(Z_{\beta ,l}^{ss}\cap \pi_l^{-1}(Z_{R_l}^s))\end{equation*}whose
fibre is $\mathbb{C}^{m_{\beta ,l}-1}\times (\mathbb{C} \setminus
\{ 0 \})$. Using next Lemma \ref{lemmafibrationEHP} we obtain
\begin{equation}\label{noseparapoly}HP_{\Stab \beta } (Y_{\beta
-l}^{\setminus E})(u,v)=HP(\mathbb{C}^{m_{\beta ,l}-1}\times
\mathbb{C}\setminus \{ 0 \})(u,v)\cdot HP_{N_l \cap \Stab \beta }
(Z_{\beta ,l}^{ss}\cap
\pi_l^{-1}(Z_{R_l}^s))(u,v).\end{equation}Now, identity
(\ref{identityDP}) tells us
that\begin{align*}HP(\mathbb{C}^{m_{\beta ,l}-1}&\times
\mathbb{C}\setminus \{ 0 \})(u,v)=(uv)^{m_{\beta
,l}}\mathcal{H}(\mathbb{C}^{m_{\beta ,l}-1}\times
\mathbb{C}\setminus \{ 0 \})(u^{-1},v^{-1})=\\& = (uv)^{m_{\beta
,l}}\cdot ( uv)^{-m_{\beta
,l}+1}((uv)^{-1}-1)=1-uv,\end{align*}so (\ref{noseparapoly}) becomes
\begin{equation}\label{noseparapoly2}HP_{\Stab \beta } (Y_{\beta
-l}^{\setminus E})(u,v)=(1-uv)\cdot HP_{N_l \cap \Stab \beta }
(Z_{\beta ,l}^{ss}\cap \pi_l^{-1}(Z_{R_l}^s))(u,v).\end{equation}

We also have that $q_{\beta}:P_{\beta}\rightarrow \Stab \beta $ is
a retraction, which implies the following isomorphism of
equivariant cohomology
\begin{equation*}\label{}H_{\Stab \beta }^{\ast} (Y_{\beta
-l}^{\setminus E})\cong H_{P_{\beta} }^{\ast} (Y_{\beta
-l}^{\setminus E}).\end{equation*}Moreover, from $\Sigma_{\beta
,l}=GY_{\beta -l}^{\setminus E}=G\times_{P_{\beta}}Y_{\beta
-l}^{\setminus E}$ we get that $H_{P_{\beta} }^{\ast} (Y_{\beta
-l}^{\setminus E})\cong H_{G}^{\ast} (\Sigma_{\beta ,l})$. These
are isomorphisms of mixed Hodge structures, hence identity
(\ref{noseparapoly2}) gives us
\begin{equation}\label{polystratodificil}HP_{G} (\Sigma_{\beta ,l})(u,v)=(1-uv)\cdot
HP_{N_l \cap \Stab \beta } (Z_{\beta ,l}^{ss}\cap
\pi_l^{-1}(Z_{R_l}^s))(u,v).\end{equation}

Then, we have the following
\begin{proposition}\label{proposicion sin numero}When $X$
is a nonsingular projective variety acted on linearly by a
reductive group $G$, the equivariant Hodge--Poincar\'e series of
$X_{(0)}^{s}$ is given by\begin{align}\label{P}HP_{G}
&(X_{(0)}^{s})(u,v)=HP_{G}
(X^{ss})(u,v)-\sum_{l=1}^{\tau}(uv)^{\lambda(R_l)}HP_{N_l}(Z_{R_l}^s)(u,v)-\\&
- \sum_{l=1}^{\tau}\sum_{\beta \in \mathcal{B}_l\setminus \{ 0 \}
}(uv)^{\lambda(\beta ,l)}(1-uv)\cdot HP_{N_l \cap \Stab \beta }
(Z_{\beta ,l}^{ss}\cap \pi_l^{-1}(Z_{R_l}^s))(u,v).\nonumber
\end{align}Moreover, the codimensions of the strata $ \Sigma_{R_l}$ are given by the following formula:
\begin{equation}\label{codimS1}\lambda (R_l)=\codim
\Sigma_{R_l}=z(l)+1\end{equation}where $\mathbb{P}^{z(l)}$ is the fibre of the projective bundle $
\mathbb{P}(\mathcal{N}_l)$ over $G\widehat{Z}^{ss}_{R_l}$, with $\mathcal{N}_l$ the normal bundle to $G\widehat{Z}^{ss}_{R_l}$
in $X^{ss}_{(R_{l-1})}$. Here $\lambda (\beta ,l)$ is the codimension of $\Sigma_{\beta ,l}$ in $X^{ss}$.
\end{proposition}
\begin{proof}Identity (\ref{P}) follows from Proposition
\ref{propositionidentity} and identities (\ref{EHPsimplestratum})
and (\ref{polystratodificil}). Regarding the codimensions of the
strata $\Sigma_{R_l}=GZ_{R_l}^s$, one has that
$$\dim GZ_{R_l}^s =\dim G-\dim N_l +\dim Z_{R_l}^s$$ where $N_l$ is the normaliser of $R_l$. Now,
$Z_{R_l}^s \subseteq \widehat{Z}_{R_l}^{ss}$ is an open subset,
therefore they have the same dimension. Moreover, one has a
locally trivial fibration $E_l \rightarrow
G\widehat{Z}_{R_l}^{ss}$ whose fibre is isomorphic to the
projective space $\mathbb{P}^{z(l)}$, so we complete the proof.\end{proof}

\begin{remark}\label{El remark que no tenia numero}\textnormal{When $G$ is connected and acts freely on
$X$, from paragraph \ref{parraequicoho} we know that
$H^{\ast}_{G}(X_{(0)}^s)\cong H^{\ast}(X_{(0)}^s/G)$, so identity
(\ref{P}) gives us the Hodge--Poincar\'e series of the quotient
$X_{(0)}^s/G$. Regarding the codimension of the strata, for those
strata $\Sigma_{(\beta ,l)}$ that can be described as
$\mathcal{S}_{\beta ,l}\setminus (E_l \cup \widehat{E}_{l-1}\cup
\ldots \cup \widehat{E}_1)$, then its codimension equals that of
${S}_{\beta ,l}$. The latter is given in (\ref{codimension de los
estratos normales}).}\end{remark}

\begin{parrafo}\textnormal{Our goal is to obtain an explicit formula
for the equivariant Hodge--Poincar\'e series of $X_{(0)}^s$ from
identity (\ref{P}).  We have fixed a maximal torus $T$, and let
$\mathcal{W}$ be the set of weights for the action of $T$ on $G$.
We have already seen that if $\beta \in \mathcal{B}$, then $\beta$ is the
closest point to 0 of
$$\Conv \{ \alpha \in \mathcal{W}:\alpha .\beta =\| \beta \|^2
\}=\Conv \{ \alpha \in \mathcal{W}:(\alpha-\beta) .\beta =0
\}.$$We have that $T$ is a maximal torus of $\Stab
\beta$. One can
define a \emph{$\beta$-sequence of length $q$} (see \cite[Section 5]{K1} for details) as a sequence
$\underline{\beta}=(\beta_1 ,\ldots ,\beta_q)$ of $q$ nonzero
elements of $\mathfrak{t}$ satisfying that for each $1\leq j\leq
q$ \begin{itemize}\item[(a)]$\beta_j$ is the closest point to 0 of
the convex hull$$\textnormal{$\Conv \{ \alpha -\beta_1 -\ldots
-\beta_{j-1}$ $ $ such that $ $ $\alpha \in \mathcal{W}$ $ $ and $
$ $\alpha .\beta_k =\| \beta_k \|^2$ $ $ for $ $ $1\leq k\leq j \}
$};$$ \item[(b)]$\beta_j$ lies in the unique Weyl chamber
containing $\mathfrak{t}_+$ of the subgroup $\bigcap_{1\leq i \leq
j}\Stab \beta_i. $
\end{itemize}
Moreover, a sequence $\underline{\beta}=(\beta_1 ,\ldots
,\beta_q)$ of $q$ nonzero elements of $\mathfrak{t}$ with $q>1$ is
a $\beta$-sequence if and only if $\beta_1 \in
\mathcal{B}\backslash \{ 0\} $ and the sequence
$\underline{\beta}'=(\beta_2 ,\ldots ,\beta_q)$ is a
$\beta$-sequence for the action of $\Stab \beta_1 $ on
$Z_{\beta_1}$. Now, for each $\beta$-sequence
$\underline{\beta}=(\beta_1 ,\ldots ,\beta_q)$, let
$T_{\underline{\beta}}$ be the subtorus of $T$ generated by the
set of weights $\{ \beta_1 ,\ldots ,\beta_q \}$, and let
$Z_{\underline{\beta}}$ be the union of the connected components
of the fixed points set of $T_{\underline{\beta}}$ on $X$.}
\end{parrafo}

\begin{parrafo}\label{el parrafo de la beta barra}\textnormal{Throughout this paper we have assumed that $X$ is a nonsingular projective
variety that is acted on by a reductive group $G$, and such that
$X^{ss}\neq X_{(0)}^s\neq \emptyset$. In the blow-up procedure we
obtain varieties $X_{(R_i)}$ for $i=1, \ldots ,\tau$ and one may
consider Morse stratifications $\{ \mathcal{S}_{\beta ,i
}\}_{\beta\in \mathcal{B}_i\backslash \{ 0\}}$ on each $X_{(R_i)}$
satisfying the properties of Paragraph \ref{Introduccion
stratification}. The groups $R_i$ are connected reductive
subgroups of $G$ that fix semistable points of $X$. Let
$\underline{\mathcal{B}_i}$ be the set of $\beta$-sequences
defined for the stratification $\{ \mathcal{S}_{\beta ,i
}\}_{\beta\in \mathcal{B}_i\backslash \{ 0\}}$.}

\textnormal{In Paragraph \ref{stratification fibra} we saw that
the stratification $\{ \mathcal{S}_{\beta ,i }\}_{\beta\in
\mathcal{B}_i\backslash \{ 0\}}$ is determined by the action of
$R_i$ on the fibres of $\mathcal{N}_i$ over
$G\widehat{Z}^{ss}_{R_i}$. Let  $\rho_i$ be the representation
of $R_i$ on the normal $\mathcal{N}_{i,x}$ to
$G\widehat{Z}^{ss}_{R_i}$ at $x$. Let
$\underline{\mathcal{B}}(\rho_i)$ be the set of $\beta$-sequences
for the representation $\rho_i$.}

\textnormal{For every $\beta$-sequence
$\underline{\beta}=(\beta_1, \ldots ,\beta_q)\in
\underline{\mathcal{B}}(\rho_i)$ we define the varieties $Z_{\beta
}^{ss}(\rho_i)$ and $Z_{\underline{\beta}}(\rho_i)$ as in the
previous section but with respect to
$\rho_i$. Let $z(\underline{\beta},i)$ be its dimension, i.e.,
$z(\underline{\beta},i)+1$ is the number, counting multiplicities,
of weights $\alpha$ such that $\alpha .\beta_j =\| \beta_j \|^2 $
for $j=1,\ldots ,q$. }

\textnormal{Let $d(\underline{\beta},i)$ be the sum over
$j=1,\ldots ,q$ of the codimension in $Z_{\beta_{j-1},i}$ of the
corresponding stratum in $Z_{\beta_{j-1},i}$ to $\beta_j$. If for
every index $j=1,\ldots ,q$ we denote by $e_j^i$ the number,
counting multiplicities, of weights $\alpha$ such that $\alpha
.\beta_k =\| \beta_k \|^2 $ for $k=1,\ldots ,j-1$ and $\alpha
.\beta_j <\| \beta_j \|^2 $, then $d(\underline{\beta},i)$ is
given by \begin{equation}\label{dimension
beta barra}d(\underline{\beta},i)=\sum_{j=1}^q
\Big{[}e_j^i-\frac{1}{2}\dim \Stab (\beta_1, \ldots ,\beta_{j-1})/\Stab
(\beta_1, \ldots ,\beta_{j})\Big{]},\end{equation}where $\Stab
\underline{\beta}=\Stab \beta_1 \cap \ldots \cap \Stab \beta_q$.
Let $q(\underline{\beta})$ be the length of the $\beta$-sequence
$\underline{\beta}$. Now, let
\begin{equation}\label{dimension w
barra}w(\underline{\beta},R_i,G)=\prod_{j=1}^q w(\beta_j,R_i\cap
\Stab (\beta_1, \ldots ,\beta_{j-1}),\Stab (\beta_1, \ldots
,\beta_{j}))\end{equation}where $w({\beta},R_i',G')$ is the number
of adjoint $R_i'$-orbits contained in the orbit of
Ad$(G')\beta$.}
\end{parrafo}

\begin{theorem}\label{teorema explicito}The $G$-equivariant Hodge--Poincar\'e series of $X_{(0)}^s$ are given by
\begin{align*}&\label{}HP_{G} (X_{(0)}^{s})(u,v)=HP_{G}
(X^{ss})(u,v)-\sum_{l=1}^{\tau}(uv)^{\lambda(R_l)}HP_{N_l}(Z_{R_l}^s)(u,v)+\\&
+ \sum_{l=1}^{\tau}
 \sum_{0\neq \underline{\beta} \in \underline{\mathcal{B}}(\rho_l)
}(-1)^{q(\underline{\beta})}(uv)^{d(\underline{\beta}
,l)}(1-(uv)^{z(\underline{\beta},l)+1})
w^{-1}(\underline{\beta},R_l,G)\cdot HP_{N_l \cap \Stab
\underline{\beta} } (Z_{R_l}^s)(u,v).\nonumber
\end{align*}Moreover, when $G$ acts freely on $X_{(0)}^{s}$ one obtains that
$HP(X_{(0)}^{s}/G)(u,v)=HP_{G} (X_{(0)}^{s})(u,v)$, so this
formula gives the Hodge--Poincar\'e series of the geometric
quotient $X_{(0)}^{s}/G$.\end{theorem}
\begin{proof}From Proposition \ref{proposicion sin numero} we have that when $X$
is a nonsingular projective variety acted on linearly by a
reductive group $G$, the equivariant Hodge--Poincar\'e series of
$X_{(0)}^{s}$ are given by\begin{align}\label{P2}HP_{G}
&(X_{(0)}^{s})(u,v)=HP_{G}
(X^{ss})(u,v)-\sum_{l=1}^{\tau}(uv)^{\lambda(R_l)}HP_{N_l}(Z_{R_l}^s)(u,v)-\\&
- \sum_{l=1}^{\tau}\sum_{\beta \in \mathcal{B}_l\setminus \{ 0 \}
}(uv)^{\lambda(\beta ,l)}(1-uv)\cdot HP_{N_l \cap \Stab \beta }
(Z_{\beta ,l}^{ss}\cap \pi_l^{-1}(Z_{R_l}^s))(u,v),\nonumber
\end{align}the codimension of the strata are given by
$\lambda (R_l)=\codim \Sigma_{R_l}=z(l)+1$, where
$\mathbb{P}^{z(l)}$ is the fibre of $\mathbb{P}(\mathcal{N}_l)$, and $\lambda (\beta
,l)=\codim \Sigma_{\beta ,l}$. For every $l=1,\ldots ,\tau$,
using induction over the length of the $\beta$-sequences in
$\underline{\mathcal{B}_l}$ we obtain the following formula
\begin{align}\label{P3}HP_{N_l \cap \Stab \beta } (Z_{ {\beta} ,l}^{ss}\cap
\pi_l^{-1}(Z_{R_l}^s))&(u,v)=HP_{N_l \cap \Stab \beta }
(Z_{{\beta} ,l}\cap \pi_l^{-1}(Z_{R_l}^s))(u,v) - \\& -
\sum_{\underline{\beta}'}
 (-1)^{q-1}(uv)^{d(\underline{\beta}',l)}\cdot HP_{N_l \cap \Stab
\underline{\beta}'} (Z_{\underline{\beta}' ,l}\cap
\pi_l^{-1}(Z_{R_l}^s))(u,v),\nonumber
\end{align}where $\underline{\beta}'$ are $\beta$-sequences of
length $q-1$ in $\underline{\mathcal{B}_l}$, and
$d(\underline{\beta}',l)$ is given by (\ref{dimension beta
barra}). Note that for $\beta$-sequences of length $1$, that is,
an element $\beta$, one has that $d(\beta ,l)=\lambda(\beta ,l)$.
Now, combining (\ref{P2}) and (\ref{P3}) we obtain the following
formula
\begin{align}&\label{P4}HP_{G} (X_{(0)}^{s})(u,v)=HP_{G}
(X^{ss})(u,v)-\sum_{l=1}^{\tau}(uv)^{\lambda(R_l)}HP_{N_l}(Z_{R_l}^s)(u,v)+\\&
+ \sum_{l=1}^{\tau}
 \sum_{0\neq \underline{\beta} \in \underline{\mathcal{B}_l}
}(-1)^{q(\underline{\beta})}(uv)^{d(\underline{\beta}
,l)}(1-uv)\cdot HP_{N_l \cap \Stab \underline{\beta} }
(Z_{\underline{\beta} ,l}\cap
\pi_l^{-1}(Z_{R_l}^s))(u,v).\nonumber
\end{align}

By \cite[Lemma 7.6 and Lemma 7.9]{K2}, each $\beta$-sequence
$\underline{\beta}$ for the Morse stratification in each
$X_{(R_l)}$ corresponds to $w(\underline{\beta},R_l,G)$
$\beta$-sequences for the stratification associated to the
representation $\rho_l$. Moreover, the fibration
(\ref{fibrationnose2}) restricted to $Z_{\underline{\beta} ,l}\cap
\pi_l^{-1}(Z_{R_l}^s)\rightarrow Z_{R_l}^s$ is a fibration with
fibre $Z_{\underline{\beta}}(\rho_l)$ which is a projective space
of dimension $z(\underline{\beta},l)$. Then, the associated
equivariant spectral sequence degenerates.
Hence, from (\ref{P4}) we get
\begin{align*}&\label{}HP_{G} (X_{(0)}^{s})(u,v)=HP_{G}
(X^{ss})(u,v)-\sum_{l=1}^{\tau}(uv)^{\lambda(R_l)}HP_{N_l}(Z_{R_l}^s)(u,v)+\\&
+ \sum_{l=1}^{\tau}
 \sum_{0\neq \underline{\beta} \in \underline{\mathcal{B}}(\rho_l)
}(-1)^{q(\underline{\beta})}(uv)^{d(\underline{\beta}
,l)}(1-uv)w^{-1}(\underline{\beta},R_l,G)\cdot
\frac{1-(uv)^{z(\underline{\beta},l)+1}}{1-uv} HP_{N_l \cap \Stab
\underline{\beta} } (Z_{R_l}^s)(u,v).\nonumber
\end{align*}
By Remark \ref{El remark que no tenia numero} we conclude the
proof of the Theorem.
\end{proof}

For a given $l\in \{ 1,\ldots ,\tau \}$, in order to compute
$HP_{N_l}(Z_{R_l}^s)$ we shall need the following slightly
different version of Lemma 1.21 of \cite{K4}. To make the notation
simpler to the eye, we set $N_l=N$ and $R_l=R$.
\begin{lemma}\label{Lemma Frances}$H^{\ast}_{N}(Z_{R}^s)$ is the invariant part of
$H^{\ast}_{N_0}(Z_{R}^s)$ under the action of the finite group
$\pi_0 N=N/N_0$, and $$H^{\ast}_{N_0}(Z_{R}^s)\cong
H^{\ast}(BR)\otimes H^{\ast}_{N_0/R}(Z_{R}^s).$$
\end{lemma}

For a given $l \in \{ 1,\ldots ,\tau \}$, when the index $\beta$
is maximal with respect to the partial order of
$\mathcal{B}_l \setminus \{ 0 \}$ one may compute the equivariant
Hodge--Poincar\'{e} polynomial of $\Sigma_{\beta ,l}$ from that of
the varieties $Z_{R_l}^s$.
\begin{lemma}\label{Lemma de los estratos}For a given $l \in \{ 1,\ldots ,\tau \}$, when the index $\beta$
is maximal with respect to the partial order given in
$\mathcal{B}_l \setminus \{ 0 \}$ the equivariant
Hodge--Poincar\'{e} polynomial of $\Sigma_{\beta ,l}$ is given
by\begin{equation*} HP_{G}(\Sigma_{\gamma})(u,v)= HP_{N_l\cap
\Stab \beta}({Z}_{R_l}^{s})(u,v)\cdot (1-(u\cdot
v)^{z(\beta ,l)+1}),\end{equation*}where $z(\beta ,l)=\dim {Z}_{\beta
,l}(\rho_l)$. Moreover, for $l=1$ one has that\begin{equation*}
HP_{G}(\Sigma_{\gamma})(u,v)= HP_{G}(\mathcal{S}_{\beta
,1})(u,v)\cdot (1-uv).\end{equation*}
\end{lemma}
\begin{proof}For the first statement, from the description
immediately after Lemma \ref{lemmafibrationEHP} one has (see
(\ref{polystratodificil}))$$HP_{G} (\Sigma_{\beta
,l})(u,v)=(1-uv)\cdot HP_{N_l \cap \Stab \beta } (Z_{\beta
,l}^{ss}\cap
\pi_l^{-1}(Z_{R_l}^s))(u,v)$$where\begin{equation}\label{eq51}Z_{\beta
,l}^{ss}\cap \pi_l^{-1}(Z_{R_l}^s)\rightarrow
Z_{R_l}^s\end{equation}is a fibration with fibre
$Z_{\beta ,l}^{ss}(\rho_l)$. When $\beta$ is a maximal element of
$\mathcal{B}_l\setminus \{ 0 \}$ for the given partial order, we
get that ${Z}_{\beta ,l}^{ss}(\rho_l)={Z}_{\beta ,l}(\rho_l)$ and
this is actually a projective space. Then the
spectral sequence associated to the fibration (\ref{eq51})
degenerates by Deligne's criterion (see \cite{D1}) and we obtain the following
isomorphism of equivariant cohomology
\begin{equation}\label{eq6}H^{\ast}_{N_l\cap \Stab \beta}(Z_{\beta ,l}^{ss}\cap
\pi_l^{-1}({Z}_{R_l}^{s}))\cong H^{\ast}_{N_l\cap \Stab
\beta}({Z}_{R_l}^{s})\otimes
H^{\ast}(\mathbb{P}^{z(\beta ,l)}),\end{equation}where
$z(\beta ,l)=\dim {Z}_{\beta ,l}(\rho_l)$. Now, (\ref{eq6}) is an
isomorphism of mixed Hodge structures, so it induces the following
identity of equivariant Hodge-Poincar\'{e} polynomials
\begin{equation}\label{eq7}HP_{N_l\cap \Stab \beta}(Z_{\beta ,l}^{ss}\cap
\pi_l^{-1}({Z}_{R_l}^{s}))(u,v)= HP_{N_l\cap \Stab
\beta}({Z}_{R_l}^{s})(u,v)\cdot \frac{1-(u\cdot
v)^{z(\beta ,l)+1}}{1-u\cdot v},\end{equation}which completes the
proof of the first statement.

Regarding the second part of the lemma, we have already pointed out that
the stratification $\{ \mathcal{S}_{\beta ,l} \}_{\beta \in
\mathcal{B}_l\backslash \{ 0\}}$ satisfies that each
$\mathcal{S}_{\beta ,l}$ retracts onto its intersection with the
exceptional divisor and if $\mathcal{B}(R_l)=\mathcal{B}(\rho_l)$,
then it retracts onto
\begin{equation}\label{eq4}G\times_{N_l\cap \Stab \beta }Z_{\beta ,l }^{ss}\cap
\pi_l^{-1}(\widehat{Z}_{R_l}^{ss})\end{equation}where
\begin{equation}\label{eq5}\pi_l :Z_{\beta ,l}^{ss}\cap
\pi_l^{-1}(\widehat{Z}_{R_l}^{ss}) \rightarrow
\widehat{Z}_{R_l}^{ss}\end{equation} is a fibration with fibre
${Z}_{\beta ,l}^{ss}(\rho_l)$. Since $\beta$ is maximal, then
${Z}_{\beta ,l}^{ss}(\rho_l)={Z}_{\beta ,l}(\rho_l)$ and this is a
projective space. For $l=1$ we have that $R_1$ has
maximum dimension among the reductive subgroups of $G$ fixing a
semistable point, then $Z_{R_1}^{ss}=Z_{R_1}^s$ and in this
particular case $Z_{R_1}^{ss}=\widehat{Z}_{R_1}^{ss}$. Hence, from
identities (\ref{eq4}) and (\ref{eq5}) we get the following
identity
\begin{equation*}H^{\ast}_G(\mathcal{S}_{\beta ,1})=H^{\ast}_{N_1\cap \Stab \beta}(Z_{\beta ,1}^{ss}\cap
\pi_1^{-1}({Z}_{R_1}^{s}))\cong H^{\ast}_{N_1\cap \Stab
\beta}({Z}_{R_1}^{s})\otimes
H^{\ast}(\mathbb{P}^{z(\beta ,1)})\end{equation*}which implies
\begin{equation*}HP_G(\mathcal{S}_{\beta ,1})(u,v)=
HP_{N_1\cap \Stab \beta}({Z}_{R_1}^{s})(u,v)\cdot \frac{1-(u\cdot
v)^{z(\beta ,1)+1}}{1-u\cdot v}.\end{equation*}Comparing this
with (\ref{polystratodificil}) and (\ref{eq7}), we finish the
proof of the lemma.
\end{proof}


\section{Cohomological formulae for the moduli space of stable vector bundles
when the rank and the degree are not coprime}\label{Cohomological formulae for the moduli space of stable vector bundles
when the rank and the degree are not coprime}


Let $\mathcal{M}(n,d)$ be the moduli space of semistable vector
bundles of rank $n$ and degree $d$ over an algebraic curve $X$ of
genus $g$. Let $\mathcal{M}^s_{(0)}(n,d)$ be the
subset of $\mathcal{M}(n,d)$ consisting of properly stable vector
bundles. It is well known that when the rank and the degree are
coprime $\mathcal{M}(n,d)=\mathcal{M}^s_{(0)}(n,d)$. In this Section we
obtain formulae for the Hodge--Poincar\'e polynomial of
$\mathcal{M}^s_{(0)}(n,d)$ when $(n,d)\neq 1$. 

We first need to
understand the corresponding stratification $\{ \Sigma_{\gamma}
\}_{\gamma \in \Gamma }$ for $\mathcal{M}(n,d)$. 
We represent $\mathcal{M}(n,d)$ as a
geometric invariant theory quotient as follows. Since tensoring by a line bundle of degree $l$ gives an
isomorphism of ${\mathcal{M}}(n,d)$ with ${\mathcal{M}}(n,d+nl)$
for any $l\in \mathbb{Z}$, there is no loss of generality in assuming $d>n(2g-1)$. Under this condition there is a natural
identification of 
$\mathcal{M}(n,d)$, and the GIT quotient $\mathcal{R}(n,d)/\!/
SL(p)$, where $p=d+n(1-g)$ and $\mathcal{R}(n,d)$ is the subset of
the set of holomorphic maps from $X$ to the Grassmannian $\Gr (n,p)
$ consisting of those holomorphic maps $h$ such that
$E_h=h^{\ast}\mathcal{T}$, where $\mathcal{T}$ is the tautological bundle, has degree $d$ and the map of sections
$\mathbb{C}^p \rightarrow H^0(X,E_h)$ induced from the quotient
bundle map $\mathbb{C}^p\times X\rightarrow E_h$ is an
isomorphism (see \cite{N} for details). 

The variety $\mathcal{R}(n,d)$ satisfies, among other properties, that if $h\in \mathcal{R}(n,d)$ then the
stabilizer of $h$ in $GL(p)$ is isomorphic to the group Aut$(E_h)$
of complex analytic automorphisms of $E_h$. Moreover,
$\mathcal{R}(n,d)$ can be embedded as a quasi-projective
subvariety of the product $\Gr (n,p)^N$ for certain integers $N$. This embedding gives us a linearisation of the action of
$SL(p)$ on $\mathcal{R}(n,d)$. 

Although $\mathcal{R}(n,d)$ is only a quasi-projective variety, this does not affect the desingularization process since the closure $\overline{\mathcal{R}(n,d)}$ of $\mathcal{R}(n,d)$ embedded in $\Gr (n,p)^N$ contains no more semistable points than $\mathcal{R}(n,d)$ does (see \cite{K4}, just before Section 3, for more details).

To find the strata $\Sigma_{\gamma}$ for $\gamma \in
\Gamma$, we need to understand how to blow up $\mathcal{M}(n,d)$
to obtain a variety $\widetilde{\mathcal{M}}(n,d)$ such that the
properly stable points are the same as the semistable ones with
respect to the action of $SL(p)$ properly linearised. We need to
blow up $\mathcal{M}(n,d)$ along a sequence of subvarieties of the
form $Z_R/\!/(N/R)$ determined by a conjugacy class $R$ of
non-trivial connected subgroups of stabilizers of semistable
points. Or what is the same, blow up $\mathcal{R}^{ss}(n,d)$ along
varieties $SL(p)Z_{R}^{ss}$ where $Z_{R}^{ss}:=\{ h\in
\mathcal{R}^{ss}(n,d)$ such that $E_h$ is fixed by $R \} \neq
\emptyset$ in decreasing order of $\dim R$ (see Subsection
\ref{StratificationSection2}). Since the central one parameter
subgroup $\mathbb{C}^{\ast}$ of $GL(p)$ acts trivially on
$\mathcal{R}(n,d)$, finding stabilizers in $GL(p)$ is essentially
equivalent to finding stabilisers in $SL(p)$. Such a subgroup $R$ of
$GL(p)$ is always the connected component of the
automorphism group Aut$E$ of a semistable vector bundle $E$.

Let $E$ be a
semistable bundle.  Let $\grad E=m_1 E_1 \oplus \ldots \oplus m_sE_s$ be the graded object associated to its Jordan--H\"older filtrations, then the $E_i$ are all properly stable bundles satisfying that
$\mu (E_i)=\mu (E)$ for all $i$ and $E_i \ncong E_j$ for all
$i\neq j$. One has that $\dim \Aut E\leqslant \dim \prod_{1\leqslant
i\leqslant s}GL(m_i)$ with equality if and only if $E\cong \grad
E$. If $E\cong \grad E$ then
\begin{equation}\label{Auto}\Aut E\cong \prod_{1\leqslant i\leqslant
s}GL(m_i).\end{equation}

In order to construct the partial desingularisation of
$\mathcal{M}(n,d)$, we need to find the
semistable vector bundles $E$ of rank $n$ and degree $d$ for which
$\dim \Aut E$ is maximal. Assume that $(n,d)=m\neq 1$ where
$n=mn'$ and $d=md'$ satisfying $(n',d')=1$. The bundles whose dimension of automorphisms is
maximal are those $E$ of the form
\begin{equation*}E=E'^{\oplus m}
\end{equation*}where $E'$ is a properly stable vector bundle of rank $n'$ and degree $d'$.
Then, the first step in the construction of
$\widetilde{\mathcal{R}}^{ss}(n,d)$ is to blow up
${\mathcal{R}}^{ss}(n,d)$ along $GL(p)Z_{GL(m)}^{ss}:=\{ h\in
{\mathcal{R}}^{ss}(n,d)$ such that $E_h \cong E'^{\oplus m}$ for
$E'\in \mathcal{M}_{(0)}^s(n',d') \}$, where
$\mathcal{M}_{(0)}^s(n',d')$ is the moduli space of properly stable
bundles of rank $n'$ and degree $d'$, $p=d+n(1-g)$ and $m=n/n'$. Let $\mathcal{R}_1(n,d)$ be
the blow-up and $\mathcal{R}_1^{ss}(n,d)$ be the semistable
stratum after that. From Paragraph \ref{Par de condiciones} (a),
$\mathcal{R}_1(n,d)\backslash \mathcal{R}_1^{ss}(n,d)$ is
isomorphic to $\phi^{-1}(\phi (GL(p)Z_{GL(m)}^{ss}))$ where $\phi
: \mathcal{R}^{ss}(n,d) \rightarrow \mathcal{R}(n,d)/\!/SL(p)$ is
the quotient map. We have that
$\phi (GL(p)Z_{GL(m)}^{ss})=\{ E\in \mathcal{M}(n,d)$ such that
$\grad E\cong E'^{\oplus m}\}$, hence $\mathcal{R}_1(n,d)\backslash
\mathcal{R}_1^{ss}(n,d)$ corresponds to the set $\{ h \in
\mathcal{R}(n,d)$ such that $\grad E_h \cong E'^{\oplus m} \}$.
Then, the first stratum is
$$\Sigma_{GL(m)}=GL(p)Z_{GL(m)}^s=GL(p)Z_{GL(m)}^{ss}$$and $\cup_{\beta \in
\mathcal{B}_1\setminus \{ 0\}}\mathcal{S}_{\beta ,1}$ is going to
be the set $\{ h \in \mathcal{R}(n,d)$ such that $\grad E_h \cong
E'^{\oplus m} \}$. The strata $\{ \Sigma_{\beta ,1} \}_{\beta \in
\mathcal{B}_{1}\setminus \{ 0 \} }$ are given by $$\Sigma_{\beta
,1}=\mathcal{S}_{\beta ,1}\setminus
\mathbb{P}(\mathcal{N}_1)$$where $\mathcal{N}_1$ is the normal
bundle corresponding to the first step in the blow-up. Before
analyzing the normal bundles let us explain how the blow-up works.

\begin{parrafo}\label{3propiedades}\textnormal{Following the previous analysis, one gets that conjugacy
classes of connected reductive subgroups of dimension less than or
equal to $m^2$ in $GL(p)$ which stabilize some point $h\in \mathcal{R}^{ss}(n,d)$ correspond to unordered sequences
$(m_1,n_1), \ldots ,(m_s,n_s)$ of pairs of positive integers
satisfying the following conditions (see \cite[Section 3]{K4} and
\cite[Section 5]{K5})
\begin{itemize}\item[(i)] $\sum_{1\leqslant j\leqslant
s}m_jn_j=n$;\item[(ii)] $\sum_{1\leqslant j\leqslant s}m_j^2 \leq
m^2$ and;\item[(iii)] $n$ divides $n_jd$ for each
$j$.\end{itemize}A representative $R$ of the conjugacy class
corresponding to $(m_1,n_1)$, ..., $(m_s,n_s)$ is given by the
image of $\prod_{1\leqslant j\leqslant s}GL(m_j)$ in $GL(p)$ given
by some fixed isomorphism of $\prod_{1\leq j\leq
s}\mathbb{C}^{m_j}\otimes \mathbb{C}^{p_j}$ with $\mathbb{C}^p$,
where $p_j=d_j +n_j(1-g)=n_jp/n$ and $d_j=n_jd/n$. Moreover, if
$N$ is the normaliser of $R$ in $GL(p)$, then its connected
component of the identity is given by
\begin{equation}\label{componente conexa}N_0\cong \prod_{1\leq
j\leq s}(GL(m_j)\times
GL(p_j))/\mathbb{C}^{\ast},\end{equation}where $\mathbb{C}^{\ast}$
is the diagonal central one paremeter subgroup of $GL(m_j)\times
GL(p_j)$, and $\pi_0 (N)=N/N_0$ is the product
\begin{equation}\label{pi0}\prod_{j\geq 0, k\geq 0}S(\sharp \{
i:m_i=j \textnormal{$ $ and $ $} n_i=k\})\end{equation}where
$S(n)$ is the symmetric group of permutations of a set of $n$
elements.}\end{parrafo}

Then, at the $k$-stage of the blow-up, there is a sequence
$(m_1,n_1)$, ..., $(m_s,n_s)$ satisfying (i), (ii) and (iii)
above, and such that a representative $R_k$ of the corresponding
conjugacy class is given by$$R_k=\prod_{1\leqslant j\leqslant
s}GL(m_j)$$embedded in $GL(p)$ as before. The variety
$GL(p)Z_{R_k}^{ss}$ is identified with the set$$\{ h\in
\mathcal{R}^{ss}(n,d):E_h\cong m_1 E_1 \oplus \ldots \oplus m_s
E_s \}$$where $E_j$ are semistable vector bundles of rank $n_j$
and degree $d_j=n_jd/n$. To obtain $\mathcal{R}_{k+1}(n,d)$ we
need to blow up $\mathcal{R}_{k}^{ss}(n,d)$ along the proper
transform of the variety $GL(p)Z_{R_k}^{ss}$ of
$\mathcal{R}^{ss}(n,d)$, then $\mathcal{R}_{k+1}(n,d)\backslash
\mathcal{R}_{k+1}^{ss}(n,d)$ will be the set of those $h \in
\mathcal{R}(n,d)$ such that $\grad E_h \cong \grad (m_1 E_1 \oplus
\ldots \oplus m_s E_s)$. Moreover, bearing in mind that
$Z_{R_k}^s$ is the set of properly stable points of $Z_{R_k}$ with
respect to the action of $N_k/R_k$, where $N_k$ is the normaliser
of $R_k$ in $GL(p)$, the stratum $\Sigma_{R_k}=GL(p)Z_{R_k}^s$ is
given by the set of $h\in \mathcal{R}^{ss}(n,d)$ such that
\begin{equation}\label{fibrados split}E_h \cong m_1 E_1 \oplus \ldots \oplus m_s
E_s\end{equation}where $E_j$ are non-isomorphic properly stable
vector bundles of rank $n_j$ and degree $d_j=n_jd/n$.

In \cite[Section 7]{AB} it is explained how to
compute the normal bundle $\mathcal{N}_k$ at a point $h$ of
$GL(p)Z_{R_k}^{ss}$. Let $GL(p)\widehat{Z}_{R_k}^{ss}$ be the proper
transform of $GL(p){Z}_{R_k}^{ss}$ in $\mathcal{R}_k^{ss}(n,d)$.
Then, the normal bundle to $GL(p)\widehat{Z}_{R_k}^{ss}$ at a
point $h$ in $\mathcal{R}_{k}^{ss}(n,d)$ such that
\begin{equation}\label{descomposition}E_h \cong m_1 E_1 \oplus \ldots \oplus m_s
E_s \end{equation}where $E_j$ are non-isomorphic properly stable
vector bundles of rank $n_j$ and degree $d_j=n_jd/n$, is
identified with $H^1(\End'_{\oplus}E_h)$, where $\End'_{\oplus}E_h
=\End E_h /\End_{\oplus }E_h$. Here $\End E_h $ is the vector
bundle of holomorphic endomorphisms of $E_h$ and $\End_{\oplus
}E_h$ is the subbundle of $\End E_h$ consisting of those
endomorphisms that preserve the decomposition
(\ref{descomposition}) (see \cite{K4} for details). One has that
$$\End'_{\oplus}E_h \cong \bigoplus_{i,j}(m_i m_j -\delta_i^j
)\Hom (E_i ,E_j)$$where $\delta_i^j$ is the Kronecker delta. Then
\begin{equation}\label{normalFibrados}H^1(\End'_{\oplus}E_h)\cong \bigoplus_{i,j=1}^s \mathbb{C}^{m_i
m_j -\delta_i^j}\otimes H^1(E_i^{\ast}\otimes E_j).\end{equation}Then,
using Riemann-Roch and bearing in mind that every
morphism between two properly stable vector bundles of the same
slope is either zero or an isomorphism, we have that\begin{equation}\label{codimension estratos
simples}\dim H^1(\End'_{\oplus}E_h)=(g-1)(n^2 - \sum_{1\leq j \leq
s}n_j^2)+\sum_{1\leq j\leq s}(m_j^2 -1),\end{equation}note that
this dimension coincides with the codimension of
$\Sigma_{R_k}=GL(p)Z_{R_k}^s$ in $\mathcal{R}^{ss}(n,d)$, that is
$\lambda (R_k)$ (see (\ref{codimS1})). Regarding the weights of the representation of $R_k$ on the normal (\ref{normalFibrados}), from \cite[Section 5]{K5} one has that these are of the form $\xi =\eta -\eta'$ where $\eta$ and $\eta'$ are weights of the standard representation of $R_k$ on $\oplus_{i=1}^s\mathbb{C}^{m_i}$.

\begin{parrafo}\label{parrafo ni me acuerdo}\textnormal{We now study the equivariant Hodge--Poincar\'e polynomial of
$\Sigma_R$ with respect to $GL(p)$ for certain $R$. Since
$GL(p)Z_{R}^s\cong GL(p) \times_N Z_R^s$ then
$H_{GL(p)}^{\ast}(\Sigma_R)\cong H_N^{\ast}(Z_R^{s})$ which is an
isomorphism of mixed Hodge structures, it is enough to compute the
equivariant Hodge--Poincar\'e polynomial of $Z_R^{s}$ with respect
to $N$. From Lemma \ref{Lemma Frances} $H_N^{\ast}(Z_R^{s})$ is
the invariant part of $H^{\ast}_{N_0}(Z_{R}^s)$ under the action
of the finite group $\pi_0 N=N/N_0$, and
$$H^{\ast}_{N_0}(Z_{R}^s)\cong H^{\ast}(BR)\otimes
H_{N_0/R}^{\ast}(Z_{R}^s).$$Note that for a given $R=\prod_{i=1}^s
GL(m_i)$ that corresponds to a sequence of pairs $(m_1,n_1)$, ...,
$(m_s,n_s)$ satisfying the properties of Paragraph
\ref{3propiedades}, one has that $$N_0/R \cong \prod_{i=1}^s
GL(p_i)/\mathbb{C}^{\ast}$$where $p_i=d_i+(1-g)n_i=n_i n/p$.
Moreover, it is known that $GL(p_i)$ acts on
$\mathcal{R}_{(0)}^s(n_i,d_i)$ in such a way that
$\mathbb{C}^{\ast}$ acts trivially and $SL(p_i)$ acts freely, then
one deduces that the action of $N_0/R$ on $Z_R^s$ is free. This
implies that \begin{equation}\label{action
libre}H_{N_0/R}^{\ast}(Z_{R}^s)\cong
H^{\ast}(Z_{R}^s/(N_0/R)).\end{equation}The Hodge--Poincar\'e
polynomial of $BR$ can be computed using
identity (\ref{BSL}). Regarding the Hodge--Poincar\'e
polynomial of $Z_{R}^s/(N_0/R)$ it can be computed as follows:  Let $N_l$ be the normaliser of $R_l$ in $GL(p)$ for every $l$, and $(N_{l})_0$ be the connected component of the
identity. In (\ref{GIT para las variedades}) we saw that
$GZ_{R_l}^{ss}/G\cong Z_{R_l}/\!/(N_l/R_l)$, then one has that
\begin{equation}\label{SigmaamgiS}\Sigma_{R_l}/G=GZ_{R_l}^{s}/G\cong
Z_{R_l}^s/(N_l/R_l).\end{equation}We describe first the varieties $Z^s_{R_l}/(N_l/R_l)$ that play a relevant role in order to understand $HP(Z_{R_{l}}^s/((N_{l})_0/R_{l}))(u,v)$. We have already seen that the highest stratum
$\Sigma_{GL(m)}=GL(p)Z_{GL(m)}^s$ corresponds to elements $h \in
\mathcal{R}^{ss}(n,d)$ such that $E_h \cong E'^{\oplus m}$ where
$E'\in \mathcal{M}_{(0)}^s(n',d')$. Then, $\Sigma_{GL(m)}/ GL(p)\cong \mathcal{M}^s_{(0)}(n',d')$ where all stabilizers of the action of $GL(p)$ belong to the conjugacy class $R_1=GL(m)$. In the second step of the blow-up, $R_2=GL(m-1)\times
\mathbb{C}^{\ast}$ and $\Sigma_{GL(m-1)\times
\mathbb{C}^{\ast}}=GL(p)Z_{GL(m-1)\times \mathbb{C}^{\ast}}^s$
consists of points $h \in \mathcal{R}^{ss}(n,d)$ such that $E_h
\cong E'^{\oplus (m-1)}\oplus E''$ where $E'$ and $E''$ are
properly stable bundles of the same rank and degree but
non-isomorphic. One has that
\begin{align*}Z_{R_2}^s/(N_{2}/R_{2})=\Sigma_{R_{2}}/G
& \cong  \big{(}( \mathcal{M}_{(0)}^s(n',d')
\times \mathcal{M}_{(0)}^s(n',d'))\backslash {\Delta}_{2}\big{)}/\pi_0(N_{2})\cong \\& \cong \big{(}( \mathcal{M}_{(0)}^s(n',d')
\times \mathcal{M}_{(0)}^s(n',d'))\backslash {\Delta}_{2}\big{)}/\mathbb{Z}/2, \nonumber \end{align*}where
${\Delta}_{2}:= \{ (E_1 , E_{2}
)\in  \mathcal{M}_{(0)}^s(n',d') \times \mathcal{M}_{(0)}^s(n',d')
$ such that $E_{1}\cong E_{2}\}$. The quotient
$\pi_0(N_{2})=N_{2}/(N_{2})_0$ acts by permuting the
factors, and equals $\mathbb{Z}/2$ (see (\ref{pi0})). }

\textnormal{In general, let $R_k$ be a representative of the conjugacy class
corresponding to a sequence $(m_1,n_1)$, ..., $(m_s,n_s)$ as in
Paragraph \ref{3propiedades}, one has that $$R_k = \prod_{1\leqslant
j\leqslant s}GL(m_j)$$properly embedded in $GL(p)$. Then, the stratum
$\Sigma_{R_k}$ consists of elements $h\in \mathcal{R}^{ss}(n,d)$
such that $E_h \cong m_1 E_1 \oplus \ldots \oplus m_s E_s $ where
$E_j$ are mutually non-isomorphic properly stable vector bundles of rank
$n_j$ and degree $d_j=n_jd/n$. In (\ref{pi0}) we pointed out that $\pi_0 (N_k)=N_k/(N_k)_0\cong \prod_{j\geq 0, r\geq 0}S(\sharp \{
i:m_i=j$ and $n_i=r\})$ where
$S(n)$ is the symmetric group of permutations of a set of $n$
elements. Then, from (\ref{SigmaamgiS}) we have that
\begin{align}\label{strato isomorfo2 abajo}Z_{R_{k}}^s/(N_{k}/R_{k})=\Sigma_{R_{k}}/G
\cong \big{(}( \mathcal{M}_{(0)}^s(n_1,d_1)
\times \ldots \times \mathcal{M}_{(0)}^s(n_s,d_s))\backslash {\Delta}_{k}\big{)}/\pi_0(N_{k}). \end{align}
Let S denote the set of collections $\{S_1,\ldots , S_h\}$ of subsets of $\{1,\ldots ,s\}$ such that $S_1 \sqcup \ldots \sqcup S_h = \{1,\ldots ,s\}$ and,
 for each $m$, $1\leq m\leq h$, we have $n_i = n_j$ whenever $i, j \in S_m$. Then $$\Delta_k = \bigsqcup_S \Delta_S$$
where $\Delta_S = \{ (E_1, \ldots , E_s) \in \mathcal{M}_{(0)}^s(n_1, d_1) \times \ldots \times \mathcal{M}_{(0)}^s(n_s, d_s)$ such that $E_i \cong E_j$ if and only if $i,j\in S_m$ for some $m$, $1\leq m\leq h\}$.
The quotient
$\pi_0(N_{k})=N_{k}/(N_{k})_0$ acts on (\ref{strato isomorfo2 abajo}) by permuting the
factors. Moreover\begin{align}\label{Para
N_0}Z_{R_{k}}^s/((N_{k})_0/R_{k})\cong ( \mathcal{M}_{(0)}^s(n_1,d_1)
\times \ldots \times \mathcal{M}_{(0)}^s(n_s,d_s))\backslash {\Delta}_{k}.\end{align}Hence, bearing in mind
that $\Delta_k$ is
a disjoint union of non-singular varieties, from
(\ref{identityDP}) and Theorem \ref{Theorem2.2} the
Hodge--Poincar\'e polynomial of
$Z_{R_{k}}^s/((N_{k})_0/R_{k})$ is given by
\begin{align}\label{HP componente conexa}HP(Z_{R_{k}}^s/((N_{k})_0/R_{k}))(u,v)&= \prod_{i=1}^s
 HP(\mathcal{M}_{(0)}^s(n_i,d_i)
)(u,v)-
(uv)^{\lambda_k}HP({\Delta}_{k})(u,v)=\\& = \prod_{i=1}^s
HP(\mathcal{M}_{(0)}^s(n_i,d_i)
)(u,v)- \sum_{S}
(uv)^{\lambda_k + \lambda_{S}}HP({\Delta}_{S})(u,v) \nonumber \end{align}where $\lambda_k$ is the
codimension of ${\Delta}_{k}$ in $\mathcal{M}_{(0)}^s(n_1,d_1) \times \ldots \times
\mathcal{M}_{(0)}^s(n_s,d_s)$, $ \lambda_{S}$ the codimension of ${\Delta}_{S}$ in $\Delta_k$.}

\textnormal{For future computations we need to understand the
equivariant cohomology group $H_S^{\ast}(Z_R^{s})$ for a subgroup
$S$ of $N$ such that $N_0 \subseteq RS$. From the latter one gets
that $$N_0/R\cong S_0/R\cap S_0$$where $N_0$ and $S_0$ are the
connected components of the identity of $N$ and $S$ respectively.
Then $H_S^{\ast}(Z_R^{s})$ is the invariant part of
$H^{\ast}_{S_0}(Z_{R}^s)$ under the action of the finite group
$\pi_0 S=S/S_0$, induced by the natural map $\pi_0 S \rightarrow
\pi_0 N$, and
\begin{equation}\label{vaya lio}H^{\ast}_{S_0}(Z_{R}^s)\cong
H^{\ast}(B(R\cap S_0))\otimes
H^{\ast}(Z_{R}^s/(N_0/R)),\end{equation}where $Z_{R}^s/(N_0/R)$ is
given by (\ref{Para N_0}).}\end{parrafo}

\begin{parrafo}\label{explicito para fibrados}\textnormal{Using Paragraph \ref{el parrafo de la beta barra}, let
$\underline{\mathcal{B}}(\rho_l)$ be the set of $\beta$-sequences
for the representation $\rho_l$ of $R_l$ on the corresponding
normal bundle, and for each $\beta$-sequence $\underline{\beta}$,
let $q(\underline{\beta})$, $d(\underline{\beta},l)$,
$z(\underline{\beta},l)$, and $w(\underline{\beta},R_l, GL(p))$ be
the positive integers defined in that paragraph. Then, Theorem \ref{teorema
explicito} tells us that the equivariant Hodge--Poincar\'e series of $\mathcal{R}_{(0)}^{s}(n,d)$
with respect to $GL(p)$ for $(n,d)\neq 1$ is given
by\begin{align}&\label{El-ultimo}HP_{GL(p)}
(\mathcal{R}_{(0)}^{s}(n,d))(u,v)=HP_{GL(p)}
(\mathcal{R}^{ss}(n,d))(u,v)-\sum_{l=1}^{\tau}(uv)^{\lambda(R_l)}HP_{N_l}(Z_{R_l}^s)(u,v)+\\&
+ \sum_{l=1}^{\tau}
 \sum_{0\neq \underline{\beta} \in \underline{\mathcal{B}}(\rho_l)
}(-1)^{q(\underline{\beta})}(uv)^{d(\underline{\beta}
,l)}(1-(uv)^{z(\underline{\beta},l)+1})
w^{-1}(\underline{\beta},R_l,G)\cdot HP_{N_l \cap \Stab
\underline{\beta} } (Z_{R_l}^s)(u,v).\nonumber
\end{align}}

\textnormal{The central one-parameter subgroup
$\mathbb{C}^{\ast}$ of $GL(p)$ acts trivially on
$\mathcal{R}_{(0)}^s(n,d)$ and such that $SL(p)$ acts freely,
then$$H^{\ast}_{GL(p)}(\mathcal{R}_{(0)}^s(n,d))\cong
H^{\ast}_{SL(p)}(\mathcal{R}_{(0)}^s(n,d)) \otimes
H^{\ast}(B\mathbb{C}^{\ast})$$and \begin{equation}\label{hpcon lo
de libre para n}HP(\mathcal{M}_{(0)}^s(n,d))(u,v)=
HP_{SL(p)}(\mathcal{R}_{(0)}^s(n,d))(u,v)=(1-uv)\cdot
HP_{GL(p)}(\mathcal{R}_{(0)}^s(n,d))(u,v),\end{equation}so from (\ref{El-ultimo}) one can obtain the Hodge--Poincar\'e polynomial of $\mathcal{M}_{(0)}^s(n,d)$.}

\textnormal{Now,  for every
$\underline{\beta}\in \underline{\mathcal{B}}(\rho_l)$ one has
that $(N_l)_0 \subseteq R_l(N_l\cap \Stab \underline{\beta})$.
Moreover, if $R_l =\prod_{j=1}^q GL(m_j)$ then $R_l \cap \Stab
\underline{\beta}=\prod_{j=1}^q (GL(m_j)\cap \Stab
\underline{\beta})$. Then, from (\ref{vaya lio}) we conclude that
$H^{\ast}_{N_l \cap \Stab \underline{\beta} } (Z_{R_l}^s)(u,v)$ is
the invariant part of$$\big{(}\bigotimes_{1\leq j\leq
s}H^{\ast}(B(GL(m_j)\cap \Stab \underline{\beta}))\big{)}\otimes
H^{\ast}(Z_{R_l}^s/((N_l)_0/R_l))$$under the action of the finite
group $\pi_0 (N_l \cap \Stab \underline{\beta} )=(N_l \cap \Stab
\underline{\beta} )/(N_l \cap \Stab \underline{\beta} )_0$,
induced by the natural map $\pi_0 (N_l \cap \Stab
\underline{\beta} ) \rightarrow \pi_0 N_l$, and
$Z_{R_l}^s/((N_l)_0/R_l)$ is given by (\ref{Para N_0}).}\end{parrafo}

\begin{parrafo}\label{Parte semistable}\textnormal{Regarding the equivariant Hodge--Poincar\'e polynomial of
$\mathcal{R}^{ss}(n,d)$ with respect to $GL(p)$, this was computed
by Earl and Kirwan in \cite{EK}. Every vector bundle $E$ of rank $n$ and degree $d$ has a strictly
ascending canonical filtration $$0=F_0 \subset F_1 \subset \ldots
\subset F_P=E$$such that the quotients $Q_j=E_j/E_{j-1}$ are
semistable and the slopes $\mu (Q_j)=\deg
(Q_j)/\rank({Q_j})=d_j'/n_j'$ satisfy that $\mu (Q_j)> \mu
(Q_{j+1})$ for every $j$. The $P$-tuple
$\overline{\mu}=(\mu (Q_1) ,\ldots ,\mu (Q_P))$ is called the \emph{type} of $E$. Let $
\overline{\mu}_0=(d/n ,\ldots ,d/n)$ and
$$d_{\overline{\mu}}=\sum_{1\leq j< i\leq P}n_i'd_j'-n_j'd_i'+n_i'n_j'(g-1).$$Then, in Theorem 1 of
\cite{EK} it is proved, among other results, that
$HP_{GL(p)}(\mathcal{R}^{ss}(n,d))(u,v)$ is given by the inductive
formula\begin{align*}HP_{GL(p)}(\mathcal{R}^{ss}&(n,d))(u,v)=\\&
=\frac{\prod_{l=1}^n(1+u^lv^{l-1})^g(1+u^{l-1}v^l)^g}
{(1-u^nv^n)\prod_{l=1}^{n-1}(1-u^lv^l)^2}-\sum_{\overline{\mu}\neq
\overline{\mu}_0}(uv)^{d_{\overline{\mu}}}\prod_{1\leq j\leq
P}HP_{GL(p)}(\mathcal{R}^{ss}(n_j',d_j'))(u,v).\nonumber\end{align*}This
formula is valid for both the cases in which $(n,d)=1$ and
$(n,d)\neq 1$. For future computations we need to know
$HP_{GL(p)}(\mathcal{R}^{ss}(2,0))(u,v)$, this is given by (see
\cite[Equation (23)]{EK}, noting the misprint in this equation)\begin{align}\label{HP de
R}HP_{GL(p)}(\mathcal{R}^{ss}(2,0))(u,v)
=\frac{(1+u)^g(1+v)^g(1+u^2v)^g(1+uv^2)^g-(uv)^{g+1}(1+u)^{2g}(1+v)^{2g}}{(1-u^2v^2)(1-uv)^2}.\end{align}Moreover,
$HP_{GL(p)}(\mathcal{R}^{ss}(n,d))(u,v)$ does not change if we
replace $d$ by $d+n\cdot z$ for any $z\in \mathbb{Z}$ (see
\cite[proof of Theorem 1]{EK}). Then, (\ref{HP de R}) gives also
the equivariant Hodge--Poincar\'e polynomial of
$\mathcal{R}^{ss}(2,d)$ with respect to $GL(p)$ for $d$
even.}\end{parrafo}


\section{Explicit computations for rank $2$ vector bundles with
even degree}\label{Cohomological formulae for the moduli space of stable vector bundles
when the rank is 2 and the degree is even}


In this section we compute explicitly
$HP(\mathcal{M}_{(0)}^s(2,d))(u,v)$ for $d$ even. Using Poincar\'e duality, from the Hodge--Poincar\'e polynomial one may obtain the Hodge--Deligne polynomial of $\mathcal{M}_{(0)}^s(2,d)$ for $d$ even. The latter was first computed in \cite{MOV2}. 

We know that $\mathcal{M}(2,d)\cong
\mathcal{R}(2,d)/\!/SL(p)$ where $p=d+2(1-g)$ and
$\mathcal{M}_{(0)}^s(2,d)\cong \mathcal{R}_{(0)}^s(2,d)/SL(p)$.
From (\ref{hpcon lo de libre para n})
\begin{equation}\label{hpcon lo de libre para
2}HP(\mathcal{M}_{(0)}^s(2,d))(u,v)=
HP_{SL(p)}(\mathcal{R}_{(0)}^s(2,d))(u,v)=(1-uv)\cdot
HP_{GL(p)}(\mathcal{R}_{(0)}^s(2,d))(u,v).\end{equation}

We need to understand the stratification $\{ \Sigma_{\gamma
}\}_{\gamma \in \Gamma}$ of $\mathcal{R}^{ss}(2,d)$ such that
$\Sigma_0 =\mathcal{R}_{(0)}^s(2,d)$. To do that we blow up
$\mathcal{R}^{ss}(2,d)$ along the subvarieties
$GL(p)\widehat{Z}_R^{ss}$ where $R$ is a representative of the
conjugacy class of all connected reductive subgroups of dimension
$\dim R$ and $\widehat{Z}_R^{ss}$ is the proper transform of
$Z_R^{ss}:=\{ h\in \mathcal{R}^{ss}(2,d)$ such that $h$ is fixed
by $R\}$ in decreasing order of dimension of $R$. The blow-up is done in two steps and the
indexing set is
$$\Gamma =\{ R_1 \} \sqcup \{ R_1 \}\times \{ \mathcal{B}_1
\backslash \{ 0 \} \} \sqcup \{ R_2 \} \sqcup \{ R_2 \}\times \{
\mathcal{B}_2 \backslash \{ 0 \} \} \sqcup \{ 0 \} ,$$where
$R_1=GL(2)$ and $R_2=T=GL(1)\times GL(1)$ which is the maximal
torus $T$ of $GL(2)$.

\begin{parrafo}\label{Primer stratum}\textnormal{The highest stratum is $\Sigma_{R_1}=
\Sigma_{GL(2)}=GZ_{GL(2)}^s$. Since $R_1=GL(2)$ has maximum
dimension among those reductive subgroups of $GL(p)$ with fixed
points in $\mathcal{R}^{ss}(2,d)$ then
$Z_{GL(2)}^s=Z_{GL(2)}^{ss}$ (see Subsection
\ref{StratificationSection2}). The conjugacy class $R_1=GL(2)$ is
embedded in $GL(p)$ using a fixed isomorphism $\mathbb{C}^2
\otimes \mathbb{C}^{1-g+d/2}\cong \mathbb{C}^p$ and letting
$GL(2)$ act on the first factor. The normaliser of $GL(2)$ in
$GL(p)$ is
$$N(GL(2))=(GL(2)\times GL(1-g+d/2))/\mathbb{C}^{\ast}$$and note
that its connected component, $N_0(GL(2))=N(GL(2))$.}

\textnormal{Then, $Z_{GL(2)}^{ss}$ is the subvariety of
$\mathcal{R}^{ss}(2,d)$ consisting of all $h\in
\mathcal{R}^{ss}(2,d)$ fixed by $GL(2)$. Hence
$$\Sigma_{GL(2)}=GZ_{GL(2)}^s=GL(p)Z_{GL(2)}^{ss}\cong
GL(p)\times_{N(GL(2))}Z_{GL(2)}^{ss}$$is the subvariety of
$\mathcal{R}^{ss}(2,d)$ of those $h$ such that $E_h\cong L\oplus
L$ for some $L\in \Jac^{d/2}$. Bearing in mind the previous
isomorphism one gets that
\begin{equation}\label{eq13}H^{\ast}_{GL(p)}(\Sigma_{GL(2)})\cong
H^{\ast}_{GL(p)}(GL(p)Z_{GL(2)}^{s})\cong
H^{\ast}_{N(GL(2))}(Z_{GL(2)}^{s}).\end{equation}Moreover, from
Lemma \ref{Lemma Frances} and (\ref{Para N_0}) one obtains
\begin{align}\label{eq13}H^{\ast}_{GL(p)}(&\Sigma_{GL(2)})
 \cong H^{\ast}(BGL(2))\otimes
H^{\ast}_{N(GL(2))/GL(2)}(Z_{GL(2)}^{s}) \cong \\& \cong
H^{\ast}(BGL(2))\otimes H^{\ast}(Z_{GL(2)}^{s}/(N(GL(2))/GL(2)))
\cong H^{\ast}(BGL(2))\otimes
H^{\ast}(\Jac^{d/2}).\nonumber\end{align}These are isomorphisms of
pure Hodge structures, so using (\ref{BSL}) they induce the
following identity of Hodge-Poincar\'{e} polynomials
\begin{align}\label{eq13}H&P_{GL(p)}(\Sigma_{GL(2)})(u,v)=HP_{GL(p)}(GL(p)Z_{GL(2)}^{s})(u,v)=\\&
= HP_{N(GL(2))}(Z_{GL(2)}^{s})(u,v)= HP(BGL(2))(u,v)\cdot
HP(\Jac^{d/2})(u,v)
=\frac{(1+u)^g(1+v)^g}{(1-uv)(1-u^2 v^2)}.\nonumber
\end{align}The codimension of $\Sigma_{GL(2)}$ in $\mathcal{R}^{ss}(2,d)$ can be computed from (\ref{codimension estratos simples}), this is
\begin{equation}\label{codimension 1}\lambda (GL(2)):=\codim \Sigma_{GL(2)}=3g.\end{equation}}\end{parrafo}

\begin{parrafo}\label{segundo stratum}\textnormal{To obtain $\sum_{\beta \in \mathcal{B}_1\backslash \{ 0 \} }
(uv)^{\lambda (\beta ,1)}{HP}_{GL(p)}(\Sigma_{\beta ,1})(u,v)$, we
first need to investigate the stratification $\{
\mathcal{S}_{\beta ,1} \}_{\beta \in \mathcal{B}_1}$ of the
variety $\mathcal{R}_1(2,d)$ -i.e., the variety obtained as a
result of the first blow-up- since $\Sigma_{\beta ,1} \cong
\mathcal{S}_{\beta ,1}\setminus E_1$ where $E_1$ is the
exceptional divisor. In order to understand the index set
$\mathcal{B}_1$ from Paragraph \ref{stratification fibra} we need
to look at the representation of $SL(2)$ on the normal
$H^1(\End_{\oplus }'E_h)$ to $GL(p)Z_{GL(2)}^{s}$ at a point $h
\in \mathcal{R}(2,d)$ such that $E_h=L\oplus L$ with $L\in
\Jac^{d/2}$. The normal $H^1(\End_{\oplus }'E_h)$ is equal to
$$H^1(\mathcal{O})\otimes \Lie(SL(2))$$ where $\Lie(SL(2))$ is the
Lie algebra of $SL(2)$. Now, $SL(2)$ acts trivially on
$H^1(\mathcal{O})$ and by conjugation on $\Lie (SL(2))$, so the
weights of the representation are $2$, $0$ and $-2$ each of them
with multiplicity $g$. This implies that there is only one weight
lying in the positive Weyl chamber, so there is only one unstable
stratum to be removed in the blow-up procedure. Let $\mathcal{S}_{\beta ,1}$
be the unique unstable stratum. From Paragraph \ref{Par de
condiciones} (a) we have that $\mathcal{S}_{\beta ,1}$ consists of elements
$h \in \mathcal{R}(2,d)$ such that $\grad E_h=L\oplus L$ with
$L\in \Jac^{d/2}$, then $\Sigma_{\beta ,1} \cong S_{\beta
,1}\setminus \mathbb{P}(H^1(\End_{\oplus }'E_h))$ consists of
elements $h \in \mathcal{R}(2,d)$ such that the vector bundle
$E_h$ is the middle term of a non-split extension
$$0\rightarrow L \rightarrow E_h \rightarrow L \rightarrow 0$$with
$L\in \Jac^{d/2}$. The index $\beta$ is maximum with respect to
the partial order in $\mathcal{B}_1\backslash \{ 0 \}$. From Lemma
\ref{Lemma de los estratos} we have that
$$HP_{GL(p)}(\Sigma_{\beta ,1})(u,v)= HP_{N(GL(2))\cap \Stab
\beta}(Z_{GL(2)}^s)(u,v)\cdot (1-(uv)^{z(\beta ,1)+1}).$$Here,
$z(\beta ,1)=g-1$ and $\Stab \beta$ is the stabiliser of $\beta \in
\mathfrak{t}_+$, where $\mathfrak{t}$ is the Lie algebra of the
maximal torus $T$ of $GL(2)$, under the adjoint action of $GL(p)$.
One has that $\Stab \beta = N_0(T)=(T\times GL(p/2))/\mathbb{C}^{\ast}$, hence
$$N(GL(2))\cap \Stab \beta \cong (T\times
GL(p/2))/\mathbb{C}^{\ast}.$$From Theorem \ref{explicito para
fibrados} one has that $H^{\ast}_{N(GL(2))\cap \Stab
\beta}(Z_{GL(2)}^s)$ is the invariant part of $$H^{\ast}(BT)\otimes
H^{\ast}(Z_{GL(2)}^s/(N_0(GL(2))/GL(2)))$$under the action of
$((T\times GL(p/2))/\mathbb{C}^{\ast})/((T\times
GL(p/2))/\mathbb{C}^{\ast})_0=$Id. Moreover, bearing in mind
identity (\ref{Para N_0}) one has that
$Z_{GL(2)}^s/(N_0(GL(2))/GL(2) )\cong \Jac^{d/2}$, hence
$$H^{\ast}_{N(GL(2))\cap \Stab \beta}(Z_{GL(2)}^s)\cong
H^{\ast}(BT)\otimes H^{\ast}(\Jac^{d/2})$$which is an isomorphism
of pure Hodge structures, then
\begin{align}\label{eq14}HP_{GL(p)}(\Sigma_{\beta
,1})&(u,v)=(1-(uv)^g)\cdot HP(BT)(u,v)\cdot
HP((Z_{GL(2)}^s/(N_0(GL(2))/GL(2)))(u,v) =\\& = (1-(uv)^g)\cdot
\frac{1}{(1-uv)^2}\cdot (1+u)^g(1+v)^g
=\frac{(1-u^gv^g)(1+u)^g(1+v)^g}{(1-u v)^2}.\nonumber
\end{align}Regarding the codimension of $\Sigma_{\beta ,1}$ in $\mathcal{R}^{ss}(2,d)$,
in Remark \ref{El remark que no tenia numero} we saw that this
coincides with the codimension of $\mathcal{S}_{\beta ,1}$. The
latter is given by (\ref{codimension de los estratos normales}),
which in this case is\begin{equation}\label{codimension 2}\lambda
(\beta ,1):=\codim \Sigma_{\beta ,1}=\codim \mathcal{S}_{\beta ,1}=n(\beta ,1)  -\dim
GL(2)/B =2g -\dim
GL(2)/B=2g-1,\end{equation}where $B$ is the Borel subgroup of
$GL(2)$ of upper triangular matrices. Note that $n(\beta ,1)=2g$ since there are only two weights $\alpha$ for the corresponding representation satisfying that $\alpha . \beta < \parallel \beta \parallel^2$ (see (\ref{codimension de los estratos normales})) and both of them have multiplicity $g$. Hence, from the previous analysis one obtains that
\begin{equation}\label{segundo estrato}\sum_{\beta \in \mathcal{B}_1\backslash \{ 0 \} }
(uv)^{\lambda (\beta ,1)}{HP}_{GL(p)}(\Sigma_{\beta
,1})(u,v)=(uv)^{2g-1}\cdot \frac{(1-u^gv^g)(1+u)^g(1+v)^g}{(1-u
v)^2}.\end{equation}}\end{parrafo}

\begin{parrafo}\label{tercer stratum}\textnormal{In the second blow-up
we consider $R_2=T=GL(1)\times GL(1)$ which is the maximal torus
$T$ of $GL(2)$ consisting of all diagonal matrices, embedded in
$GL(p)$ using the above embedding of $GL(2)$ in $GL(p)$. Then
$$N(T)=(N^T\times GL(p/2))/\mathbb{C}^{\ast}$$ where $N^T$ is the
normaliser of $T$ in $GL(2)$. Now
\begin{equation}\label{Sigma}\Sigma_T\cong GL(p){Z}_{T}^{s} \cong
GL(p)\times_{N(T)}{Z}_{T}^{s}\end{equation}is the subvariety of
$\mathcal{R}^{ss}(2,d)$ consisting of all $h$ such that $E_h\cong
L_1 \oplus L_2$ where $L_1 \ncong L_2$ and $L_i \in \Jac^{d/2}$.
Bearing in mind (\ref{Sigma}), the polynomial
$HP_{GL(p)}(\Sigma_{T})(u,v)$ is the same as
$HP_{N(T)}(Z_T^s)(u,v)$. The latter is computed in the following
lemma.}
\begin{lemma}\label{Lemma sin numero}\begin{align*}HP_{N(T)}(Z_T^s)(u,v)&=(1-uv)^{-1}(1-u^2v^2)^{-1}(\frac{1}{2}(1+u)^{2g}(1+v)^{2g}(1+uv)
+\\&
+\frac{1}{2}(1-u^2)^g(1-v^2)^g(1-uv)-(uv)^g(1+u)^g(1+v)^g).\end{align*}
\end{lemma}
\begin{proof}By Lemma \ref{Lemma Frances} we know that $H^{\ast}_{N(T)}(Z_T^s)$ is the invariant part of $H^{\ast}(BT)\otimes H^{\ast}(Z_T^s/(N_0(T)/T))$ under the action of
the finite group $\pi_0 (N(T))=N(T)/N_0(T)=N^T/T\cong \mathbb{Z}/2$. Moreover, by 
(\ref{Para
N_0}) one has that $Z_T^s/(N_0(T)/T)\cong \Jac^{d/2}\times
\Jac^{d/2}\setminus \Delta$ where $\Jac^{d/2}$ is the Jacobian of degree $d/2$ and $\Delta$  is the diagonal of $\Jac^{d/2}\times
\Jac^{d/2}$.

Then, the Hodge--Poincar\'e polynomial $HP_{N(T)}(Z_T^s)(u,v)$ is given by\begin{equation}\label{Polinomio suma}HP^+(BT)(u,v)HP^+(\Jac^{d/2}\times
\Jac^{d/2}\setminus \Delta)(u,v)+HP^-(BT)(u,v)HP^-(\Jac^{d/2}\times
\Jac^{d/2}\setminus \Delta)(u,v)\end{equation}where the subscripts $+$ and $-$ refer to the corresponding eigenspaces of eigenvalues $+1$ and $-1$ for the action of $\mathbb{Z}/2$ in both $H^{\ast}(BT)$ and $H^{\ast}(\Jac^{d/2}\times
\Jac^{d/2}\setminus \Delta)$ respectively. Moreover, the Hodge--Poincar\'e polynomial of
$\Jac^{d/2} \times \Jac^{d/2} \setminus \Delta$ is given by
(\ref{HP componente conexa}). The codimension
of $\Delta \cong \Jac^{d/2}$ in $\Jac^{d/2} \times \Jac^{d/2}$ is $g$, then$$HP(\Jac^{d/2} \times \Jac^{d/2} \setminus \Delta)(u,v)=HP(\Jac^{d/2} \times \Jac^{d/2})(u,v)-(uv)^gHP(\Jac^{d/2})(u,v),$$hence$$HP^+(\Jac^{d/2} \times \Jac^{d/2} \setminus \Delta)(u,v)=HP^+(\Jac^{d/2} \times \Jac^{d/2})(u,v)-(uv)^gHP^+(\Jac^{d/2})(u,v),$$and the same is satisfied for $HP^-(\Jac^{d/2} \times \Jac^{d/2} \setminus \Delta)(u,v)$.
We have that $H^{\ast}(BT)\cong H^{\ast}(BGL(1))\otimes H^{\ast}(BGL(1))$ and $H^{\ast}(\Jac^{d/2}\times \Jac^{d/2} )\cong H^{\ast}(\Jac^{d/2})\otimes H^{\ast}(\Jac^{d/2})$. These are isomorphisms of mixed Hodge structures, actually these are isomorphisms of pure Hodge structures. The action of the non-trivial element of $\mathbb{Z}/2$ on \begin{equation*}H^{p,q}(\Jac^{d/2}\times \Jac^{d/2} )\cong H^{p}(\Jac^{d/2})\otimes H^{q}(\Jac^{d/2})\cong \bigoplus_{p_1+p_2=p, q_1+q_2=q}H^{p_1,q_1}(\Jac^{d/2})\otimes H^{p_2,q_2}(\Jac^{d/2})
\end{equation*}is given by $a\otimes b \in H^{p_1,q_1}(\Jac^{d/2})\otimes H^{p_2,q_2}(\Jac^{d/2})$ goes to $(-1)^{(p_1+q_1)(p_2+q_2)}b\otimes a$. Analogously for $H^{\ast}(BT)$. Note that $\mathbb{Z}/2$ acts trivially on the diagonal.

One has that $$HP^+(\Jac^{d/2} \times \Jac^{d/2})(u,v)=\sum_{p,q}(-1)^{p+q}\dim \Sym (H^{p,q}(\Jac^{d/2} \times \Jac^{d/2}))u^pv^q$$where $\Sym$ denotes the symmetric part, and $HP^-(\Jac^{d/2} \times \Jac^{d/2})(u,v)=\sum_{p,q}(-1)^{p+q}\dim \Ant (H^{p,q}(\Jac^{d/2} \times \Jac^{d/2}))u^pv^q$ where $\Ant$ denotes the antisymmetric part. The same is satisfied for $BT$.

Now, when $(p_1,q_1)\neq (p_2,q_2)$ one has that $\dim \Sym (H^{p,q}(\Jac^{d/2} \times \Jac^{d/2}))=\dim \Ant (H^{p,q}(\Jac^{d/2} \times \Jac^{d/2}))=\frac{1}{2}\dim (H^{p,q}(\Jac^{d/2} \times \Jac^{d/2}))$. When $(p_1,q_1)= (p_2,q_2)$ elements of the form $a\otimes a$ for $a\in H^{p_1,q_1}(\Jac^{d/2})$ also need to be considered. Then, when $p_1+q_1$ is even these elements are symmetric, and when $p_1+q_1$ is odd they are antisymmetric. Hence, it is satisfied that $$\dim \Sym (H^{2p_1,2q_1}(\Jac^{d/2}\times \Jac^{d/2}))-\dim \Ant (H^{2p_1,2q_1}(\Jac^{d/2}\times \Jac^{d/2}))=(-1)^{p_1+q_1}\dim H^{p_1,q_1}(\Jac^{d/2}),$$and $$\dim \Sym (H^{2p_1,2q_1}(\Jac^{d/2}\times \Jac^{d/2}))+\dim \Ant (H^{2p_1,2q_1}(\Jac^{d/2}\times \Jac^{d/2}))=\dim H^{2p_1,2q_1}(\Jac^{d/2}\times \Jac^{d/2}).$$Then\begin{align*}HP^+&(\Jac^{d/2} \times \Jac^{d/2})(u,v)=\sum_{p,q}(-1)^{p+q}\dim \Sym (H^{p,q}(\Jac^{d/2} \times \Jac^{d/2}))u^pv^q=\\& =
\sum_{\substack{(p_1,q_1)\neq (p_2,q_2)\\ p_1+p_2=p\\ q_1+q_2=q}}(-1)^{p_1+q_1+p_2+q_2}\frac{1}{2}\dim (H^{p_1,q_1}(\Jac^{d/2}) \otimes H^{p_2,q_2}(\Jac^{d/2}))u^{p_1+q_1}v^{p_2+q_2}+\nonumber \\& +\sum_{\substack{2 p_1=p\\ 2q_1=q}}\frac{1}{2}\big{[}(\dim H^{p_1,q_1}(\Jac^{d/2}) )^2u^{p_1+q_1}v^{p_2+q_2}+(-1)^{p_1+q_1}\dim H^{p_1,q_1}(\Jac^{d/2})u^{2p_1}v^{2q_1}\big{]}\nonumber =\\&
= \frac{1}{2}[HP(\Jac^{d/2})(u,v)]^2+\frac{1}{2}HP(\Jac^{d/2})(-u^2,-v^2).\nonumber\end{align*}Bearing in mind that $\mathbb{Z}/2$ acts trivially on the diagonal, one gets that \begin{align}\label{mas}HP^+&(\Jac^{d/2} \times \Jac^{d/2}\setminus \Delta)(u,v)=\\& =\frac{1}{2}[HP(\Jac^{d/2})(u,v)]^2+\frac{1}{2}HP(\Jac^{d/2})(-u^2,-v^2)-(uv)^gHP(\Jac^{d/2})(u,v)=\nonumber \\& =\frac{1}{2}(1+u)^{2g}(1+v)^{2g}+\frac{1}{2}(1-u^2)^g(1-v^2)^g-(uv)^g(1+u)^g(1+v)^g.\nonumber \end{align}Analogously\begin{align}\label{menos}HP^-(\Jac^{d/2} \times \Jac^{d/2}\setminus \Delta)(u,v)&=\frac{1}{2}[HP(\Jac^{d/2})(u,v)]^2-\frac{1}{2}HP(\Jac^{d/2})(-u^2,-v^2) = \\& =\frac{1}{2}(1+u)^{2g}(1+v)^{2g}-\frac{1}{2}(1-u^2)^g(1-v^2)^g.\nonumber \end{align}Repeating the previous argument for $H^{\ast}(BT)$ but taking into account that $H^{p,q}(BT)\neq 0$ only when $p=q$, we obtain \begin{equation}\label{masmenos}\textnormal{$HP^+(BT)(u,v)=\frac{1}{(1-uv)(1-u^2v^2)}$ $ $ $ $ and $ $ $ $ $HP^-(BT)(u,v)=\frac{uv}{(1-uv)(1-u^2v^2)}.$}\end{equation}Finally, from (\ref{Polinomio suma}), (\ref{mas}), (\ref{menos}), and (\ref{masmenos}) we conclude.\end{proof}
\textnormal{The codimension of $\Sigma_{T}$ in $\mathcal{R}^{ss}(2,d)$ can be
computed from (\ref{codimension estratos simples}), this is
\begin{equation}\label{codimension T}\lambda (T):=\codim \Sigma_{T}=2g-2.\end{equation}}\end{parrafo}

\begin{parrafo}\label{cuarto stratum}\textnormal{Regarding $\sum_{\beta \in \mathcal{B}_2\backslash \{ 0 \} }
(uv)^{\lambda (\beta ,2)}{HP}_{GL(p)}(\Sigma_{\beta ,2})(u,v)$ we
have to look at the representation of $T$ on the normal
$H^1(\End_{\oplus }'E)$ to $GL(p)\widehat{Z}_{T}^{ss}$ at a point
$h\in \mathcal{R}^{ss}(2,d)$ such that $E_h=L_1\oplus L_2$ where
$L_1 \ncong L_2 $ and $L_i \in \Jac^{d/2}$. Here
$$H^1(\End_{\oplus }'E)=H^1(\Hom (L_1 ,L_2 ))\oplus H^1(\Hom (L_2 ,L_1
)),$$a diagonal matrix diag$(t_1 ,t_2)$ acts as multiplication by
$t_1 t_2^{-1}$ on the first factor and by $t_2t_1^{-1}$ on the
second. Then the weights of the representation of $T\cap SL(2)$
are $2$ and $-2$ each with multiplicity $g-1$. This implies that
again there is only one unstable stratum $\mathcal{S}_{\beta ,2}$
to be removed. Hence the index $\beta$ of this stratum is maximum
for the given partial order of $\mathcal{B}_2$. From Paragraph
\ref{Par de condiciones} (a) we know that this stratum is the
proper transform of the set of all $h \in \mathcal{R}^{ss}(2,d)$
such that $\grad E_h\cong L_1 \oplus L_2$ where $L_i \in
\Jac^{d/2}$. Hence $\Sigma_{\beta ,2}=\mathcal{S}_{\beta
,2}\backslash \mathbb{P}(H^1(\Hom (L_1 ,L_2 ))\oplus H^1(\Hom (L_2
,L_1 )))$ consists of $h \in \mathcal{R}^{ss}(2,d)$ such that
$E_h$ is the middle term of a non-split extension
$$0\rightarrow L_1 \rightarrow E_h \rightarrow L_2 \rightarrow
0$$with $L_1\ncong L_2$ and $L_i\in \Jac^{d/2}$. Again, from Lemma
\ref{Lemma de los estratos} we have that
$$HP_{GL(p)}(\Sigma_{\beta ,2})(u,v)= HP_{N(T)\cap \Stab
\beta}(Z_{T}^s)(u,v)\cdot (1-(uv)^{z(\beta ,2)+1}).$$Here,
$z(\beta ,2)=g-2$ and
$N(T)\cap \Stab \beta \cong (T\times
GL(p/2))/\mathbb{C}^{\ast}$. From Theorem \ref{explicito para
fibrados} one has that $H^{\ast}_{N(T)\cap \Stab \beta}(Z_{T}^s)$
is the invariant part of
$$H^{\ast}(BT)\otimes
H^{\ast}(Z_{T}^s/(N_0(T)/T))$$under the action of $((T\times
GL(p/2))/\mathbb{C}^{\ast})/((T\times
GL(p/2))/\mathbb{C}^{\ast})_0=$Id. Moreover, by (\ref{Para N_0})
we have that $Z_{T}^s/(N_0(T)/T)\cong \Jac^{d/2} \times \Jac^{d/2}
\setminus \Delta$ where $\Delta$ is the diagonal in $\Jac^{d/2}
\times \Jac^{d/2}$. The Hodge--Poincar\'e polynomial of
$\Jac^{d/2} \times \Jac^{d/2} \setminus \Delta$ is given by
(\ref{HP componente conexa}). Bearing in mind that the codimension
of $\Delta$ in $\Jac^{d/2} \times \Jac^{d/2}$ is $g$, one obtains
\begin{align}HP_{GL(p)}(\Sigma_{\beta
,2})&(u,v)=(1-(uv)^{g-1})\cdot HP(BT)(u,v)\cdot
HP(Z_{T}^s/(N_0(T)/T))(u,v) =\\&
 = (1-(uv)^{g-1})
((1+u)^{2g}(1+v)^{2g}-(uv)^g(1+u)^{g}(1+v)^{g})(1-uv)^{-2}.\nonumber
\end{align}Regarding the codimension of $\Sigma_{\beta ,2}$ in $\mathcal{R}^{ss}(2,d)$,
in Remark \ref{El remark que no tenia numero} we saw that this
coincides with the codimension of $\mathcal{S}_{\beta ,2}$. The
latter is given by (\ref{codimension de los estratos normales}),
which in this case is\begin{equation}\label{codimension 2}\lambda
(\beta ,2):=\codim \Sigma_{\beta ,2}=\codim \mathcal{S}_{\beta ,2}=n(\beta ,2) -\dim
T/B=g-1 -\dim
T/B=g-1\end{equation}where $B$ is a Borel subgroup of $T$. Hence,
one obtains that
\begin{align}\label{segundo2 estrato}\sum_{\beta \in \mathcal{B}_2\backslash \{ 0 \}
}& (uv)^{\lambda (\beta ,2)}{HP}_{GL(p)}(\Sigma_{\beta
,2})(u,v)=\\& =((uv)^{g-1}-(uv)^{2g-2})
\frac{\big{(}(1+u)^{2g}(1+v)^{2g}-(uv)^g(1+u)^{g}(1+v)^{g}\big{)}}{(1-uv)^{2}}.\nonumber\end{align}}\end{parrafo}

From the previous analysis we obtain the following Theorem.
\begin{theorem}The Hodge--Poincar\'e polynomial of $\mathcal{M}^s_{(0)}(2,d)$ for $d$ even is
given
by\begin{align*}HP(&\mathcal{M}^s_{(0)}(2,d))(u,v)=\frac{1}{2(1-uv)(1-u^2v^2)}\bigg{[}2(1+u)^g(1+v)^g
(1+u^2v)^g(1+uv^2)^g-\\&
-(uv)^{g-1}(1+u)^{2g}(1+v)^{2g}(2-(uv)^{g-1}+(uv)^{g+1})-(uv)^{2g-2}(1-u^2)^g(1-v^2)^g(1-uv)^2\bigg{]}.\end{align*}\end{theorem}
\begin{proof}From Proposition \ref{proposicion sin numero}, Lemma
\ref{Lemma sin numero} and identities (\ref{HP de R}),
(\ref{eq13}), (\ref{codimension 1}), (\ref{segundo estrato}),
(\ref{codimension T}), and (\ref{segundo2 estrato}) we have that
\begin{align*}HP&_{GL(p)}(\mathcal{R}^s_{(0)}(2,d))(u,v)=\frac{1}{2(1-uv)^2(1-u^2v^2)}\bigg{[}2(1+u)^g(1+v)^g
(1+u^2v)^g(1+uv^2)^g-\\&
-(uv)^{g-1}(1+u)^{2g}(1+v)^{2g}(2-(uv)^{g-1}+(uv)^{g+1})-(uv)^{2g-2}(1-u^2)^g(1-v^2)^g(1-uv)^2\bigg{]}.\end{align*}
Bearing in mind identity (\ref{hpcon lo de libre para 2}), we
conclude.\end{proof}

\begin{remark}\textnormal{Since $\mathcal{M}^s_{(0)}(2,d)$ is a smooth variety, from identity
(\ref{identityDP}) we have that the
Hodge--Deligne polynomial of $\mathcal{M}^s_{(0)}(2,d)$ is given
by
$$\mathcal{H}(\mathcal{M}^s_{(0)}(2,d))(u,v)=(uv)^{4g-3}\cdot HP(\mathcal{M}^s_{(0)}(2,d))(u^{-1},v^{-1}),$$then
\begin{align*}\mathcal{H}(\mathcal{M}^s_{(0)}(2,d))(u,v)
=& \frac{1}{2(1-uv)(1-u^2v^2)}[
2(1+u)^g(1+v)^g(1+u^2v)^g(1+uv^2)^g- \\& -
(1+u)^{2g}(1+v)^{2g}(1+2u^{g+1}v^{g+1}-u^2v^2)-(1-u^2)^g(1-v^2)^g(1-uv)^2],
\end{align*}this polynomial was obtained independently by Mu\~noz, Ortega and V\'azquez--Gallo (see \cite[Theorem 5.2]{MOV2}) using
the relation of $\mathcal{M}^s_{(0)}(2,d)$ with certain moduli
spaces of triples.}\end{remark}

\subsection*{Acknowledgements}I would like to thank Prof. Montserrat Teixidor i Bigas and Prof. Peter E. Newstead for their support, encouragement and advice during the realization of my Ph.D. dissertation, of which this paper is part. I also would like to thank Prof. Frances C. Kirwan for very helpful conversations and advice, and for reading preliminary versions of this work. Thank you to the departments of mathematics of the Universities of Liverpool, Oxford and Tufts.

\end{document}